\documentclass[12pt]{article}

\usepackage[margin=.8in]{geometry}

\usepackage{amsmath}
\usepackage{amssymb}
\usepackage{amsthm}
\usepackage{amsfonts}
\usepackage{enumerate}
\usepackage{graphicx}
\usepackage{url}
\usepackage[table]{xcolor}
\usepackage{soul}
\usepackage{enumitem}

\usepackage{subcaption}

\usepackage{ytableau}

\usepackage{tikz}
\usetikzlibrary{arrows,shapes,positioning,decorations.markings,decorations.pathreplacing}
\tikzset{->-/.style={decoration={
  markings,
  mark=at position .5 with {\arrow{>}}},postaction={decorate}}}

\usepackage{array}
\newcolumntype{?}{!{\vrule width 3pt}}

\newtheorem{theorem}{Theorem}[section]

\newtheorem{corollary}[theorem]{Corollary}
\newtheorem{proposition}[theorem]{Proposition}

\title{Colored Motzkin Paths\\of Higher Order\footnote{Research supported by NSF Grant DMS-1559912.}}

\author{
Isaac DeJager\\
\footnotesize LeTourneau University\\[-0.8ex] 
\footnotesize Longview, Texas, U.S.A.\\[-0.8ex]
\footnotesize\tt IsaacDeJager@letu.edu
\and
Madeleine Naquin\\
\footnotesize Spring Hill College \\[-0.8ex]
\footnotesize Mobile, Alabama, U.S.A.\\[-0.8ex]
\footnotesize\tt madeleine.c.naquin@email.shc.edu
\and
Frank Seidl\\
\footnotesize University of Michigan\\[-0.8ex]
\footnotesize  Ann Arbor, Michigan, U.S.A. \\[-0.8ex]
\footnotesize\tt fcseidl@umich.edu
\and
Paul Drube\\
\footnotesize Valparaiso University\\[-0.8ex] 
\footnotesize Valparaiso, Indiana, U.S.A. \\[-0.8ex]
\footnotesize\tt paul.drube@valpo.edu
}

\begin{document}

\maketitle

\begin{abstract}
Motzkin paths of order-$\ell$ are a generalization of Motzkin paths that use steps $U=(1,1)$, $L=(1,0)$, and $D_i=(1,-i)$ for every positive integer $i \leq \ell$.  We further generalize order-$\ell$ Motzkin paths by allowing for various coloring schemes on the edges of our paths.  These $(\vec{\alpha},\vec{\beta})$-colored Motzkin paths may be enumerated via proper Riordan arrays, mimicking the techniques of Aigner in his treatment of ``Catalan-like numbers".  After an investigation of their associated Riordan arrays, we develop bijections between $(\vec{\alpha},\vec{\beta})$-colored Motzkin paths and a variety of well-studied combinatorial objects.  Specific coloring schemes $(\vec{\alpha},\vec{\beta})$ allow us to place $(\vec{\alpha},\vec{\beta})$-colored Motzkin paths in bijection with different subclasses of generalized $k$-Dyck paths, including $k$-Dyck paths that remain weakly above horizontal lines $y=-a$, $k$-Dyck paths whose peaks all have the same height modulo-$k$, and Fuss-Catalan generalizations of Fine paths. A general bijection is also developed between $(\vec{\alpha},\vec{\beta})$-colored Motzkin paths and certain subclasses of $k$-ary trees.
\end{abstract}

\section{Introduction}
\label{sec: intro}

A Motzkin path of length $n$ and height $m$ is an integer lattice path from $(0,0)$ to $(n,m)$ that uses the step set $\lbrace U = (1,1), L=(1,0), D=(1,-1) \rbrace$ and remains weakly above the horizontal line $y=0$.  We denote the set of all such Motzkin paths by $\mathcal{M}_{n,m}$ and let $\vert \mathcal{M}_{n,m} \vert = M_{n,m}$.  The cardinalities $M_{n,0} = M_n$ correspond to the Motzkin numbers, a well-known integer sequence that begins (for $n \geq 0$) as $1, 1, 2, 4, 9, 21, 51, 127, \hdots$.  For more information about the Motzkin numbers and their various combinatorial interpretations, see Aigner \cite{Aigner1}, Bernhart \cite{Bernhart}, and Donaghey and Shapiro \cite{DonS}. 

For any $\alpha,\beta \geq 0$, an element of $\mathcal{M}_{n,m}$ is said to be an $(\alpha,\beta)$-colored Motzkin path of length $n$ and height $m$ if each of its $L$ steps at height $y=0$ is labeled by one of $\alpha$ colors and each of its $L$ steps at height $y > 0$ is labeled by one of $\beta$ colors.  By the ``height" of a step in a lattice path we mean the $y$-coordinate of its right endpoint.  We denote the set of all $(\alpha,\beta)$-colored Motzkin paths by $\mathcal{M}_{n,m}(\alpha,\beta)$ and let $\vert \mathcal{M}_{n,m}(\alpha,\beta) \vert = M_{n,m}(\alpha,\beta)$.  By analogy with above, for fixed $\alpha,\beta$ we henceforth refer to the integer sequences $\lbrace M_{n,0}(\alpha,\beta) \rbrace_{n=0}^\infty$ as the $(\alpha,\beta)$-colored Motzkin numbers.  See Figure \ref{1fig: Colored Motzkin Paths Example} for an illustration of the set $\mathcal{M}_{3,0}(1,2)$.  That example establishes our convention of using positive integers for our ``colors".

\begin{figure}[ht!]
\centering
\begin{tikzpicture}
    [scale=.7, auto=left, every node/.style = {circle, fill=black, inner sep=1.25pt}]
    \draw[thick, dotted, color = black] (0,0) to (3,0);
    \node(0*) at (0,0) {};
    \node(1*) at (1,0) {};
    \node(2*) at (2,0) {};
    \node(3*) at (3,0) {};
    \draw[thick] (0*) to node[midway,above,fill=none,font=\small] {1} (1*);
    \draw[thick] (1*) to node[midway,above,fill=none,font=\small] {1} (2*);
    \draw[thick] (2*) to node[midway,above,fill=none,font=\small] {1} (3*);
\end{tikzpicture}
\hspace{.15in}
\begin{tikzpicture}
	[scale=.7,auto=left,every node/.style={circle, fill=black,inner sep=1.25pt}]
	\draw[thick,dotted,color=black] (0,0) to (3,0);
	\node (0*) at (0,0) {};
	\node (1*) at (1,0) {};
	\node (2*) at (2,1) {};
	\node (3*) at (3,0) {};
	\draw[thick] (0*) to node[midway,above,fill=none,font=\small] {1} (1*);
	\draw[thick] (1*) to (2*);
	\draw[thick] (2*) to (3*);
\end{tikzpicture}
\hspace{.15in}
\begin{tikzpicture}
	[scale=.7,auto=left,every node/.style={circle, fill=black,inner sep=1.25pt}]
	\draw[thick,dotted,color=black] (0,0) to (3,0);
	\node (0*) at (0,0) {};
	\node (1*) at (1,1) {};
	\node (2*) at (2,1) {};
	\node (3*) at (3,0) {};
	\draw[thick] (0*) to (1*);
	\draw[thick] (1*) to node[midway,above,fill=none,font=\small] {1} (2*);
	\draw[thick] (2*) to (3*);
\end{tikzpicture}
\hspace{.15in}
\begin{tikzpicture}
	[scale=.7,auto=left,every node/.style={circle, fill=black,inner sep=1.25pt}]
	\draw[thick,dotted,color=black] (0,0) to (3,0);
	\node (0*) at (0,0) {};
	\node (1*) at (1,1) {};
	\node (2*) at (2,1) {};
	\node (3*) at (3,0) {};
	\draw[thick] (0*) to (1*);
	\draw[thick] (1*) to node[midway,above,fill=none,font=\small] {2} (2*);
	\draw[thick] (2*) to (3*);
\end{tikzpicture}
\hspace{.15in}
\begin{tikzpicture}
	[scale=.7,auto=left,every node/.style={circle, fill=black,inner sep=1.25pt}]
	\draw[thick,dotted,color=black] (0,0) to (3,0);
	\node (0*) at (0,0) {};
	\node (1*) at (1,1) {};
	\node (2*) at (2,0) {};
	\node (3*) at (3,0) {};
	\draw[thick] (0*) to (1*);
	\draw[thick] (1*) to (2*);
	\draw[thick] (2*) to node[midway,above,fill=none,font=\small] {1} (3*);
\end{tikzpicture}
    \caption{All $(1,2)$-colored Motzkin paths in the set $\mathcal{M}_{3,0}(1,2)$.}
    \label{1fig: Colored Motzkin Paths Example}
\end{figure}
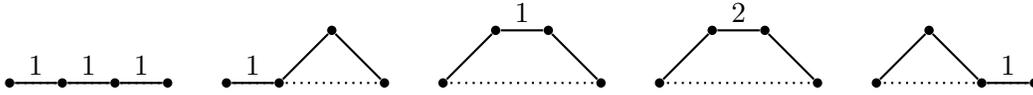

Our definition of colored Motzkin paths specializes to the ``$k$-colored Motzkin paths" of Barrucci et al. \cite{BDPP} or Sapounakis and Tsikouras \cite{ST1,ST2} when $\alpha = \beta = k$.  Our definition is also distinct from the $(u,l,d)$-colored Motzkin paths of Woan \cite{Woan1,Woan2} or Mansour, Schork and Sun \cite{MSS}.  In those papers, $U$, $L$, and $D$ steps are all colored and there is no distinction between $L$ steps at various heights.  Our particular generalization of Motzkin paths has been chosen to be in alignment with the Riordan array-oriented methodologies of Aigner \cite{Aigner2}.

It is clear that $M_{n,m}(\alpha,\beta) = 0$ unless $0 \leq m \leq n$, as well as that $M_{0,0}(\alpha,\beta) = 1$.  For $n \geq 1$, the cardinalities $M_{n,m}(\alpha,\beta)$ may be computed via the recursion of Proposition \ref{1thm: Motzkin triangle recursion}, variations of which already appear elsewhere.

\begin{proposition}
\label{1thm: Motzkin triangle recursion}
For all $n \geq 1$ and $0 \leq m \leq n$,
$$M_{n,m}(\alpha,\beta) =
\begin{cases}
M_{n-1,m-1}(\alpha,\beta) + \beta \kern+1pt M_{n-1,m}(\alpha,\beta) + M_{n-1,m+1}(\alpha,\beta) & \text{if } m \geq 1; \\[4pt]
\alpha \kern+1pt M_{n-1,0}(\alpha,\beta) + M_{n-1,1}(\alpha,\beta) & \text{if } m=0.
\end{cases}$$
\end{proposition}
\begin{proof}
For $m \geq 1$, partition the paths of $\mathcal{M}_{n,m}(\alpha,\beta)$ according to their final step.  The subset of those paths that end with a $U$ step are in bijection with $\mathcal{M}_{n-1,m-1}(\alpha,\beta)$, those that end with a $D$ step are in bijection with $\mathcal{M}_{n-1,m+1}(\alpha,\beta)$, and those that end with a $L$ step of a fixed color are in bijection with $\mathcal{M}_{n-1,m}(\alpha,\beta)$, for each of the $\beta$ colors.  The $m=1$ case is similar, although here the elements of $\mathcal{M}_{n,m}(\alpha,\beta)$ cannot end in a $U$ step and the final $L$ steps may carry $\alpha$ possible colors. 
\end{proof}

Using Proposition \ref{1thm: Motzkin triangle recursion}, one may define an infinite lower-triangular array whose $(n,m)$-entry is $M_{n,m}(\alpha,\beta)$.  See Figure \ref{1fig: colored Motzkin triangle} for the first four rows of this triangle, which we refer to as the $(\alpha,\beta)$-colored Motzkin triangle.  Notice that the upper-leftmost nonzero entry of this triangle corresponds to $(n,m) = (0,0)$.

\begin{figure}[ht!]
\begin{center}
\def\arraystretch{2}
\begin{tabular}{l l l l}
$1$ & & & \\ [-10pt]
$\alpha$ & $1$ & & \\ [-10pt]
$\alpha^2+1$ & $\alpha + \beta$ & $1$ & \\ [-10pt]
$\alpha^3 + 2\alpha + \beta $ & $\alpha^2 + \alpha \beta + \beta^2 + 2$ & $\alpha + 2 \beta$ & $1$ \\
\end{tabular}
\end{center}

\vspace{-14pt}

\caption{The first four nonzero rows of the $(\alpha,\beta)$-colored Motzkin triangle.}
\label{1fig: colored Motzkin triangle}
\end{figure}

We will be interested primarily in the leftmost nonzero column of the $(\alpha,\beta)$-Motzkin triangle.  Ranging over various $(\alpha,\beta)$, these are the $(\alpha,\beta)$-colored Motzkin numbers and constitute one class of what Aigner \cite{Aigner2} refers to as ``Catalan-like numbers".  Shown in Table \ref{1tab: Catalan-like numbers} are the sequences associated with all $0 \leq \alpha,\beta \leq 4$, many of which are well-represented in the literature.

\begin{table}[ht!]
\centering
\def\arraystretch{1.3}
\scalebox{.9}{
\begin{tabular}{|c|c|c|c|c|c|}
\hline
 & \textbf{$\boldsymbol{\beta=0}$} & \textbf{$\boldsymbol{\beta=1}$} & \textbf{$\boldsymbol{\beta=2}$} & \textbf{$\boldsymbol{\beta=3}$} & \textbf{$\boldsymbol{\beta=4}$} \\ \hline
\textbf{$\boldsymbol{\alpha=0}$} & A126120 & Riordan numbers ($R_n$) & Fine numbers ($F_n$) & A1177641 & A185132 \\ \hline
\textbf{$\boldsymbol{\alpha=1}$} & A001405 & Motzkin numbers ($M_n$) & Catalan numbers ($C_n$) & A033321 & - \\ \hline
\textbf{$\boldsymbol{\alpha=2}$} & A054341 & A005773 & Catalan numbers ($C_{n+1}$) & A007317 & A033543 \\ \hline
\textbf{$\boldsymbol{\alpha=3}$} & A126931 & A059738 & $\binom{2n+1}{n+1}$ & A002212 & A064613 \\ \hline
\textbf{$\boldsymbol{\alpha=4}$} & - & - & A049027 & A026378 & A005572 \\ \hline
\end{tabular}}
\caption{Integer sequences corresponding to the $(\alpha,\beta)$-colored Motzkin numbers, for various choices of $(\alpha,\beta)$.  Numbered entries correspond to OEIS \cite{OEIS}.  Dashes correspond to sequences that currently do not appear on OEIS.}
\label{1tab: Catalan-like numbers}
\end{table}

As with the Catalan-like numbers of Aigner \cite{Aigner2}, our notion of $(\alpha,\beta)$-colored Motzkin paths is most efficiently recast within the language of proper Riordan arrays.  Let $d(t),h(t)$ be a pair of formal power series such that $d(0) \neq 0$, $h(0)=0$ and $h'(0) \neq 0$.  The proper Riordan array $\mathcal{R}(d(t),h(t))$ associated with these power series is an infinite, lower-triangular array whose $(i,j)$-entry is $d_{i,j} = [t^i]d(t)h(t)^j$.  Here was adopt the standard convention where $[t^i]p(t)$ denotes the coefficient of the $t^i$ in the power series $p(t)$.  See Rogers \cite{Rogers} or Merlini et al. \cite{MRSV} for more background information on Riordan arrays. 

Fundamental to the theory of Riordan arrays is the fact that every proper Riordan array $\mathcal{R}(d(t),h(t))$ is uniquely determined by a pair of power series $A(t) = \sum_{i=0}^\infty a_i t^i$ and $Z(t) = \sum_{i=0}^\infty z_i t^i$ such that 

\begin{equation}
\label{1eq: A and Z sequence recursion}
d_{i,j}=
\begin{cases}
a_0 \kern+1pt d_{i-1,j-1} + a_1 \kern+1pt d_{i-1,j} + a_2 \kern+1pt d_{i-1,j+1} + \hdots & \text{for all } j \geq 1; \\[4pt]
z_0 \kern+1pt d_{i-1,0} + z_1 \kern+1pt d_{i-1,1} + z_2 \kern+1pt d_{i-1,2} + \hdots & \text{for } j=0.
\end{cases}
\end{equation}

These series are referred to as the $A$-sequence and $Z$-sequence of $\mathcal{R}(d(t),h(t))$, respectively.  It may be shown that the $A$- and $Z$-sequences of $\mathcal{R}(d(t),h(t))$ satisfy

\begin{equation}
\label{1eq: A and Z sequences vs d(t),h(t)}
h(t) = t \kern+1pt A(h(t)), \hspace{.5in} d(t) = \frac{d(0)}{1-t \kern+1pt Z(h(t))}.
\end{equation}

With this terminology in hand, Proposition \ref{1thm: Motzkin triangle recursion} immediately guarantees that the $(\alpha,\beta)$-colored Motzkin triangle is a proper Riordan array for every choice $\alpha,\beta \geq 0$.  In particular, the $(\alpha,\beta)$-colored Motzkin triangle is the proper Riordan array with $A$-sequence $A(t) = 1 + \beta \kern+1pt t + t^2$ and $Z$-sequence $Z(t) = \alpha + t$.

\subsection{Outline of Paper}
\label{subsec: outline of paper}

The goal of this paper is to adapt the aforementioned phenomena to ``higher-order" Motzkin paths, a generalization of traditional Motzkin paths whose step set includes a down step $D_i = (1,-i)$ for every positive integer $1 \leq i \leq \ell$ up to some fixed upper bound $\ell$.  Section \ref{sec: higher-order motzkin paths} defines the relevant notion of coloring for higher-order Motzkin paths, describes the proper Riordan arrays that enumerate these colored paths, and proves a series of general results about those proper Riordan arrays.  Section \ref{sec: combinatorial interpretations} then introduces a series of combinatorial interpretations for colored higher-order Motzkin paths that directly generalize the combinatorial interpretations suggested by Table \ref{1tab: Catalan-like numbers}. In particular, colored higher-order Motzkin paths are placed in bijection with various classes of generalized $k$-Dyck paths, entirely new generalizations of Fine paths, $k$-Dyck paths whose peaks only occur at specific heights, and various subsets of $k$-ary trees.  Appendix \ref{sec: appendix} closes the paper by comparing the first columns of our proper Riordan arrays against sequences in OEIS \cite{OEIS} for various ``easy" colorations, providing impetus for future investigations.

\section{Higher-Order Motzkin Paths}
\label{sec: higher-order motzkin paths}

Fix $\ell \geq 1$.  An \textbf{order-$\ell$ Motzkin path} of length $n$ and height $m$ is an integer lattice path from $(0,0)$ to $(n,m)$ that uses step set $\lbrace U = (1,1), D_0 = (1,0), D_1 = (1,-1), \hdots , D_{\ell} = (1,-\ell) \rbrace$ and remains weakly above the horizontal line $y = 0$.  We denote the set of all such paths by $\mathcal{M}_{n,m}^\ell$, and let $\vert \mathcal{M}_{n,m}^\ell \vert = M_{n,m}^\ell$.  Note that order-$1$ Motzkin paths correspond to the traditional notion of Motzkin paths, so that $M_{n,0}^1 = M_{n,0}$ are the Motzkin numbers.

For $\ell > 1$, our notion of order-$\ell$ Motzkin paths are distinct from the ``higher-rank" Motzkin paths studied by Mansour, Schork and Sun \cite{MSS} or Sapounakis and Tsikouras \cite{ST1}. In particular, our order-$\ell$ Motzkin paths don't allow for up steps of multiple slopes.  In the limit of $\ell \rightarrow \infty$, our order-$\ell$ Motzkin paths correspond to the \L ukasiewicz paths investigated by Cheon, Kim and Shapiro \cite{CKS}.

We now look to color order-$\ell$ Motzkin paths in a way that directly generalizes the Riordan array properties of Section \ref{sec: intro}.  So fix $\ell \geq 1$, and let $\vec{\alpha} = (\alpha_0,\hdots,\alpha_{\ell-1})$, $\vec{\beta} = (\beta_0,\hdots,\beta_{\ell-1})$ be any pair of $\ell$-tuples of non-negative integers.  An \textbf{$(\vec{\alpha},\vec{\beta})$-colored Motzkin path} is an element of $\mathcal{M}_{n,m}^\ell$ where each $D_i$ step that ends at height $y=0$ is labeled by one of $\alpha_i$ colors, and each $D_i$ step that ends at height $y>0$ is labeled by one of $\beta_i$ colors, for each $0 \leq i \leq \ell-1$.  We denote the set of all $(\vec{\alpha},\vec{\beta})$-colored Motzkin paths by $\mathcal{M}_{n,m}^\ell(\vec{\alpha},\vec{\beta})$ and let $\vert \mathcal{M}_{n,m}^\ell(\vec{\alpha},\vec{\beta}) \vert = M_{n,m}^\ell(\vec{\alpha},\vec{\beta})$.  For fixed $\ell,\vec{\alpha},\vec{\beta}$, we refer to the integer sequences $\lbrace M_{n,0}^\ell (\vec{\alpha},\vec{\beta}) \rbrace_{n=0}^\infty$ as the \textbf{$(\vec{\alpha},\vec{\beta})$-colored Motzkin numbers} of order-$\ell$.

See Figure \ref{2fig: higher-order colored Motzkin paths example} for an example of order-$2$ Motzkin paths.  Notice that the only steps which fail to receive colors are $U$ steps and down steps $D_\ell$ of maximal negative slope.

\begin{figure}[ht!]
\centering
\begin{tikzpicture}
    [scale=.55, auto=left, every node/.style = {circle, fill=black, inner sep=1.25pt}]
    \draw[thick, dotted, color = black] (0,0) to (3,0);    
    \node(0*) at (0,0) {};
    \node(1*) at (1,0) {};
    \node(2*) at (2,0) {};
    \node(3*) at (3,0) {};    
    \draw[thick] (0*) to (1*) to node[midway,above,fill=none,font=\small] {1} (0*);
    \draw[thick] (1*) to (2*) to node[midway,above,fill=none,font=\small] {1} (1*);
    \draw[thick] (2*) to (3*) to node[midway,above,fill=none,font=\small] {1} (2*);
\end{tikzpicture}
\hspace{.22in}
\begin{tikzpicture}
    [scale=.55, auto=left, every node/.style = {circle, fill=black, inner sep=1.25pt}]
    \draw[thick, dotted, color = black] (0,0) to (3,0);    
    \node(0*) at (0,0) {};
    \node(1*) at (1,0) {};
    \node(2*) at (2,1) {};
    \node(3*) at (3,0) {};    
    \draw[thick] (0*) to (1*) to node[midway,above,fill=none,font=\small] {1} (0*);
    \draw[thick] (1*) to (2*);
    \draw[thick] (2*) to (3*) to node[midway,above,fill=none,font=\small] {1} (2*);
\end{tikzpicture}
\hspace{.22in}
\begin{tikzpicture}
    [scale=.55, auto=left, every node/.style = {circle, fill=black, inner sep=1.25pt}]
    \draw[thick, dotted, color = black] (0,0) to (3,0);
    \node(0*) at (0,0) {};
    \node(1*) at (1,0) {};
    \node(2*) at (2,1) {};
    \node(3*) at (3,0) {};
    \draw[thick] (0*) to (1*) to node[midway,above,fill=none,font=\small] {1} (0*);
    \draw[thick] (1*) to (2*);
    \draw[thick] (2*) to (3*) to node[midway,above,fill=none,font=\small] {2} (2*);
\end{tikzpicture}
\hspace{.22in}
\begin{tikzpicture}
    [scale=.55, auto=left, every node/.style = {circle, fill=black, inner sep=1.25pt}]
    \draw[thick, dotted, color = black] (0,0) to (3,0);
    \node(0*) at (0,0) {};
    \node(1*) at (1,1) {};
    \node(2*) at (2,0) {};
    \node(3*) at (3,0) {};
    \draw[thick] (0*) to (1*);
    \draw[thick] (1*) to (2*) to node[midway,above,fill=none,font=\small] {1} (1*);
    \draw[thick] (2*) to (3*) to node[midway,above,fill=none,font=\small] {1} (2*);
\end{tikzpicture}
\hspace{.22in}
\begin{tikzpicture}
    [scale=.55, auto=left, every node/.style = {circle, fill=black, inner sep=1.25pt}]
    \draw[thick, dotted, color = black] (0,0) to (3,0);
    \node(0*) at (0,0) {};
    \node(1*) at (1,1) {};
    \node(2*) at (2,0) {};
    \node(3*) at (3,0) {};
    \draw[thick] (0*) to (1*);
    \draw[thick] (1*) to (2*) to node[midway,above,fill=none,font=\small] {2} (1*);
    \draw[thick] (2*) to (3*) to node[midway,above,fill=none,font=\small] {1} (2*);
\end{tikzpicture}
\hspace{.22in}
\begin{tikzpicture}
    [scale=.55, auto=left, every node/.style = {circle, fill=black, inner sep=1.25pt}]
    \draw[thick, dotted, color = black] (0,0) to (3,0);
    \node(0*) at (0,0) {};
    \node(1*) at (1,1) {};
    \node(2*) at (2,1) {};
    \node(3*) at (3,0) {};
    \draw[thick] (0*) to (1*);
    \draw[thick] (1*) to (2*) to node[midway,above,fill=none,font=\small] {1} (1*);
    \draw[thick] (2*) to (3*) to node[midway,above,fill=none,font=\small] {1} (2*);
\end{tikzpicture}

\vspace{.3in}

\begin{tikzpicture}
    [scale=.55, auto=left, every node/.style = {circle, fill=black, inner sep=1.25pt}]
    \draw[thick, dotted, color = black] (0,0) to (3,0);
    \node(0*) at (0,0) {};
    \node(1*) at (1,1) {};
    \node(2*) at (2,1) {};
    \node(3*) at (3,0) {};
    \draw[thick] (0*) to (1*);
    \draw[thick] (1*) to (2*) to node[midway,above,fill=none,font=\small] {2} (1*);
    \draw[thick] (2*) to (3*) to node[midway,above,fill=none,font=\small] {1} (2*);
\end{tikzpicture}
\hspace{.22in}
\begin{tikzpicture}
    [scale=.55, auto=left, every node/.style = {circle, fill=black, inner sep=1.25pt}]
    \draw[thick, dotted, color = black] (0,0) to (3,0);
    \node(0*) at (0,0) {};
    \node(1*) at (1,1) {};
    \node(2*) at (2,1) {};
    \node(3*) at (3,0) {};
    \draw[thick] (0*) to (1*);
    \draw[thick] (1*) to (2*) to node[midway,above,fill=none,font=\small] {3} (1*);
    \draw[thick] (2*) to (3*) to node[midway,above,fill=none,font=\small] {1} (2*);
\end{tikzpicture}
\hspace{.22in}
\begin{tikzpicture}
    [scale=.55, auto=left, every node/.style = {circle, fill=black, inner sep=1.25pt}]
    \draw[thick, dotted, color = black] (0,0) to (3,0);
    \node(0*) at (0,0) {};
    \node(1*) at (1,1) {};
    \node(2*) at (2,1) {};
    \node(3*) at (3,0) {};
    \draw[thick] (0*) to (1*);
    \draw[thick] (1*) to (2*) to node[midway,above,fill=none,font=\small] {1} (1*);
    \draw[thick] (2*) to (3*) to node[midway,above,fill=none,font=\small] {2} (2*);
\end{tikzpicture}
\hspace{.22in}
\begin{tikzpicture}
    [scale=.55, auto=left, every node/.style = {circle, fill=black, inner sep=1.25pt}]
    \draw[thick, dotted, color = black] (0,0) to (3,0);
    \node(0*) at (0,0) {};
    \node(1*) at (1,1) {};
    \node(2*) at (2,1) {};
    \node(3*) at (3,0) {};
    \draw[thick] (0*) to (1*);
    \draw[thick] (1*) to (2*) to node[midway,above,fill=none,font=\small] {2} (1*);
    \draw[thick] (2*) to (3*) to node[midway,above,fill=none,font=\small] {2} (2*);
\end{tikzpicture}
\hspace{.22in}
\begin{tikzpicture}
    [scale=.55, auto=left, every node/.style = {circle, fill=black, inner sep=1.25pt}]
    \draw[thick, dotted, color = black] (0,0) to (3,0);
    \node(0*) at (0,0) {};
    \node(1*) at (1,1) {};
    \node(2*) at (2,1) {};
    \node(3*) at (3,0) {};
    \draw[thick] (0*) to (1*);
    \draw[thick] (1*) to (2*) to node[midway,above,fill=none,font=\small] {3} (1*);
    \draw[thick] (2*) to (3*) to node[midway,above,fill=none,font=\small] {2} (2*);
\end{tikzpicture}
\hspace{.22in}
\begin{tikzpicture}
    [scale=.55, auto=left, every node/.style = {circle, fill=black, inner sep=1.25pt}]
    \draw[thick, dotted, color = black] (0,0) to (3,0);
    \node(0*) at (0,0) {};
    \node(1*) at (1,1) {};
    \node(2*) at (2,2) {};
    \node(3*) at (3,0) {};
    \draw[thick] (0*) to (1*);
    \draw[thick] (1*) to (2*);
    \draw[thick] (2*) to (3*);
\end{tikzpicture}
\caption{All paths in $\mathcal{M}_{3,0}^2 (\vec{\alpha},\vec{\beta})$ with $\vec{\alpha} = (1,2)$, $\vec{\beta} = (3,3)$.}
\label{2fig: higher-order colored Motzkin paths example}
\end{figure}
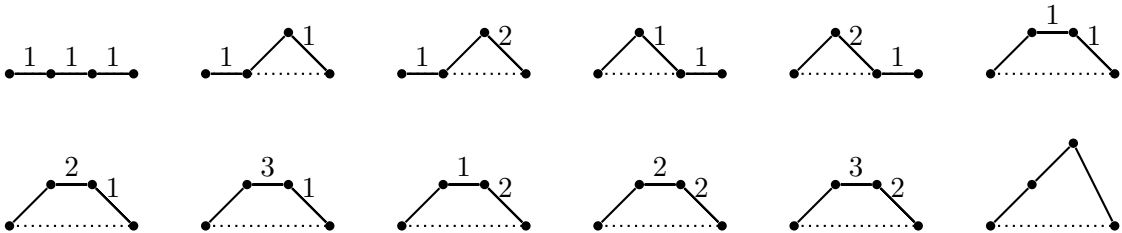

It is once again clear that $M_{n,m}^\ell(\vec{\alpha},\vec{\beta})=0$ unless $0 \leq m \leq n$, as well as that $M_{0,0}^\ell (\vec{\alpha},\vec{\beta}) = 1$.  For any pair $\vec{\alpha},\vec{\beta}$, we may then assemble an infinite, lower-triangular array whose $(n,m)$ entry is $M_{n,m}^\ell(\vec{\alpha},\vec{\beta})$.  We call this triangle the \textbf{$(\vec{\alpha},\vec{\beta})$-colored Motzkin triangle}.

As with the $(\alpha,\beta)$-colored Motzkin triangle of Section \ref{sec: intro}, the $(\vec{\alpha},\vec{\beta})$-colored Motzkin triangle is a proper Riordan array for every choice of $\vec{\alpha}$ and $\vec{\beta}$:

\begin{proposition}
\label{2thm: higher-order Motzkin triangle as Riordan array}
Fix $\ell \geq 1$, and let $\vec{\alpha} = (\alpha_0,\hdots,\alpha_{\ell-1})$, $\vec{\beta} = (\beta_0,\hdots,\beta_{\ell-1})$ be any pair of $\ell$-tuples of non-negative integers.  Then the $(\vec{\alpha},\vec{\beta})$-colored Motzkin triangle is a proper Riordan array with $A$- and $Z$-sequences

$$A(t) = 1 + \beta_0 \kern+1pt t + \hdots + \beta_{\ell-1} \kern+1pt t^\ell + t^{\ell+1};$$

\vspace{-.15in}

$$Z(t) = \alpha_0 + \alpha_1 \kern+1pt t + \hdots + \alpha_{\ell-1} \kern+1pt t^{\ell-1} + t^\ell.$$

\end{proposition}
\begin{proof}
In a manner similar to the proof of Proposition \ref{1thm: Motzkin triangle recursion}, we demonstrate the recurrences

\vspace{.1in}

\noindent
\scalebox{.85}{
$M_{n,m}^\ell(\vec{\alpha},\vec{\beta}) = 
\begin{cases}
M_{n-1,m-1}^\ell(\vec{\alpha},\vec{\beta}) + \beta_0 \kern+1pt M_{n-1,m}^\ell(\vec{\alpha},\vec{\beta}) + \hdots + \beta_{\ell-1} \kern+1pt M_{n-1,m+\ell-1}^\ell(\vec{\alpha},\vec{\beta}) + M_{n-1,m+\ell}^\ell(\vec{\alpha},\vec{\beta}) & \text{if } m \geq 1; \\[6pt]
\alpha_0 \kern+1pt M_{n-1,0}^\ell(\vec{\alpha},\vec{\beta}) + \alpha_1 \kern+1pt M_{n-1,1}^\ell(\vec{\alpha},\vec{\beta}) \hdots + \alpha_{\ell-1} \kern+1pt M_{n-1,\ell-1}^\ell(\vec{\alpha},\vec{\beta}) + M_{n-1,\ell}^\ell(\vec{\alpha},\vec{\beta}) & \text{if } m =1.
\end{cases}$}

\vspace{.1in}

For $m \geq 1$, we partition the paths of $\mathcal{M}_{n,m}^\ell(\vec{\alpha},\vec{\beta})$ according to their final step.  Those paths that end with a $U$ step are in bijection with $\mathcal{M}_{n-1,m-1}^\ell(\vec{\alpha},\vec{\beta})$, and those that end with a $D_i$ step of a fixed color are in bijection with $\mathcal{M}_{n-1,m+i}^\ell(\vec{\alpha},\vec{\beta})$, for every $0 \leq i \leq \ell$ and for each of the $\beta_i$ colors of that $i$.  The $m=1$ case is similar, although here the elements of $\mathcal{M}_{n,m}^\ell(\vec{\alpha},\vec{\beta})$ cannot end with a $U$ step and the $D_i$ steps can carry one of $\alpha_i$ possible colors.
\end{proof}

Proposition \ref{2thm: higher-order Motzkin triangle as Riordan array} may be used to quickly generate elements of the $(\vec{\alpha},\vec{\beta})$-colored Motzkin triangle.  In the order-$\ell$ case, the $(\vec{\alpha},\vec{\beta})$-colored Motzkin numbers of order-$\ell$ constitute a $2\ell$-dimensional array of integer sequences that can be compared to previously-studied results, as in Table \ref{1tab: Catalan-like numbers}.  Appendix \ref{sec: appendix} presents a series of tables that test the resulting sequences against OEIS \cite{OEIS}, for a variety of ``nice" choices of $\vec{\alpha},\vec{\beta}$ in the $\ell=2$ case.

\subsection{General Properties of $(\vec{\alpha},\vec{\beta})$-Colored Motzkin Triangles}
\label{subsec: general properties of higher-order Motzkin triangles}

We begin by proving a number of general identities involving the $(\vec{\alpha},\vec{\beta})$-colored Motzkin numbers and their associated Riordan arrays.  Much of what follows is most easily cast in terms of generating functions.  As such, for any $m,\ell,\vec{\alpha},\vec{\beta}$ we define the ordinary generating function $M^\ell_m(\vec{\alpha},\vec{\beta},t) = \sum_{n=0}^\infty M_{n,m}^\ell(\vec{\alpha},\vec{\beta}) \kern+1pt t^n$.

Our first result relates the $m \geq 1$ columns of the $(\vec{\alpha},\vec{\beta})$-colored Motzkin triangle to the $m=0$ column, letting us characterize the $d(t),h(t)$ series of the associated Riordan arrays.    

\begin{proposition}
\label{2thm: height-m Motzkin paths as height-0 convolutions}
Fix $\ell \geq 1$ and take any pair of $\ell$-tuples of non-negative integers $\vec{\alpha},\vec{\beta}$.  For every $m \geq 1$,

$$M_m^\ell(\vec{\alpha},\vec{\beta},t) = t \kern+2pt M_0^\ell(\vec{\alpha},\vec{\beta},t) \kern+1pt M_{m-1}^\ell(\vec{\beta},\vec{\beta},t) = t^m \kern+2pt M_0^\ell(\vec{\alpha},\vec{\beta},t) \kern+1pt M_0^\ell(\vec{\beta},\vec{\beta},t)^m.$$
\end{proposition}
\begin{proof}
Every height-$m$ path may be decomposed into a height-$0$ path and a height-$(m-1)$ path as shown below, with the intermediate $U$ step in that image being the rightmost $U$ step that ends at height $1$.  The labels inside the boxes denote the colorings applicable to each subpath, with the second coloration changing because none of its steps terminate at an overall height of $0$.
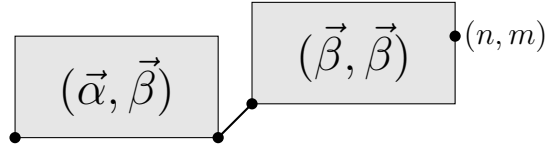
\begin{center}
\begin{tikzpicture}
	[scale=.45,auto=left,every node/.style={circle,fill=black,inner sep=0pt,outer sep=0pt}]
	\draw[fill=gray!20!](0,0) rectangle (6,3);
	\node[fill=none] (b1) at (3,1.5) {\scalebox{1.5}{$(\vec{\alpha},\vec{\beta})$}};
	\draw[thick] (6,0) to (7,1);
	\draw[fill=gray!20!](7,1) rectangle (13,4);
	\node[fill=none] (b1) at (10,2.5) {\scalebox{1.5}{$(\vec{\beta},\vec{\beta})$}};	
	\node[inner sep=1.5pt] (1) at (0,0) {};
	\node[inner sep=1.5pt] (2) at (6,0) {};
	\node[inner sep=1.5pt] (3) at (7,1) {};
	\node[inner sep=1.5pt] (4) at (13,3) {};
	\node[inner sep=0pt,fill=none] (5) at (14.5,3) {$(n,m)$};
\end{tikzpicture}
\end{center}

The decomposition implied above immediately demonstrates the first equality.  The second equality follows from repeated application of the first equality.
\end{proof}

\begin{corollary}
\label{2thm: colored Motzkin triangles as Riordan arrays}
For any $\ell \geq 1$ and any pair of $\ell$-tuples of non-negative integers $\vec{\alpha},\vec{\beta}$, the $(\vec{\alpha},\vec{\beta})$-colored Motzkin triangle is the proper Riordan array $\mathcal{R}(d(t),h(t))$ with $d(t) = M_0^\ell(\vec{\alpha},\vec{\beta},t)$ and $h(t) = t \kern+2pt M_0^\ell (\vec{\beta},\vec{\beta},t)$.
\end{corollary}
\begin{proof}
From Proposition \ref{2thm: height-m Motzkin paths as height-0 convolutions}, we see that the $j^{th}$-column of the $(\vec{\alpha},\vec{\beta})$-colored Motzkin triangle has generating function $t^j \kern+2pt M_0^\ell(\vec{\alpha},\vec{\beta},t) \kern+1pt M_0^\ell(\vec{\beta},\vec{\beta},t)^j = M_0^\ell(\vec{\alpha},\vec{\beta},t) \kern+1pt \left( t \kern+2pt M_0^\ell(\vec{\beta},\vec{\beta},t) \right)^j$.
\end{proof}

Unfortunately, Corollary \ref{2thm: colored Motzkin triangles as Riordan arrays} is only of practical use if we have an explicit formula for the generating functions $M_0^\ell(\vec{\alpha},\vec{\beta},t)$, and such a formula will not be attempted here.  Still of conceptual interest is the well-known fact that the row-sums of the Riordan array $\mathcal{R}(d(t),h(t))$ has generating function $d(t) / (1-h(t))$.  As the row-sums of the $(\vec{\alpha},\vec{\beta})$-colored Motzkin triangle enumerate $(\vec{\alpha},\vec{\beta})$-Motzkin paths of length $n$ and any height $m \geq 0$, we have:

\begin{corollary}
\label{2thm: row-sums of Motzkin triangle}
For any $\ell \geq 1$ and any pair of $\ell$-tuples of non-negative integers $\vec{\alpha},\vec{\beta}$, the generating function for $(\vec{\alpha},\vec{\beta})$-colored Motzkin paths of length $n$ and any height $m \geq 0$ is
$$\sum_{m=0}^\infty M_m^\ell(\vec{\alpha},\vec{\beta},t) = \frac{M_0^\ell(\vec{\alpha},\vec{\beta},t)}{1-t \kern+2pt M_0^\ell(\vec{\beta},\vec{\beta},t)} = M_0^\ell(\vec{\alpha},\vec{\beta},t) \left( 1 + t \kern+2pt M_0^\ell(\vec{\beta},\vec{\beta},t) + t^2 \kern+2pt M_0^\ell(\vec{\beta},\vec{\beta},t)^2 + \hdots \right).$$
\end{corollary}

Of greater practical importance is the alternative characterization of row-sums presented below, which applies only when $\vec{\alpha} = \vec{\beta}$.  In Section \ref{sec: combinatorial interpretations}, this result will allow us to immediately translate every combinatorial interpretations for $(\vec{\alpha},\vec{\alpha})$-colored Motzkin numbers into an associated combinatorial interpretation for $(\vec{\alpha}+\hat{e}_1,\vec{\alpha})$-colored Motzkin numbers.

\begin{theorem}
\label{2thm: row-sums of Motzkin triangle, color change version}
Fix $\ell \geq 1$ and take any $\ell$-tuple of non-negative integers $\vec{\alpha}$.  If $\hat{e}_1 = (1,0,\hdots,0)$, the generating function for $(\vec{\alpha},\vec{\alpha})$-colored Motzkin paths of length $n$ and any height $m \geq 0$ is

\vspace{-.15in}

$$\sum_{m=0}^\infty M_m^\ell(\vec{\alpha},\vec{\alpha},t) = M_0^\ell (\vec{\alpha}+\hat{e}_1,\vec{\alpha},t).$$
\end{theorem}
\begin{proof}
We construct a bijection $\phi_n$ from $S = \bigcup_{m=0}^\infty \mathcal{M}_{n,m}^\ell(\vec{\alpha},\vec{\alpha})$ to $\mathcal{M}_{n,0}^\ell(\vec{\alpha}+\hat{e}_1,\vec{\alpha})$, for arbitrary $n \geq 0$.  So take any path $P \in S$, and assume that $P$ has height $m$.  Then $P$ contains precisely $m$ up steps that are ``visible" from the right, meaning that they are the rightmost $U$ steps at their particular height.  Replacing these $U$ steps with (temporarily-uncolored) level steps yields a path with $m$ uncolored $D_0$ steps at height $y=0$.  Coloring these $D_0$ steps with a new color $\alpha_0+1$ results in a unique element $\phi_n(P) \in \mathcal{M}_{n,0}^\ell(\vec{\alpha}+\hat{e}_1,\vec{\alpha})$.

See Figure \ref{2fig: row-sum bijection example} for an example of this map $\phi_n$.  This process is clearly invertible.  As the only $D_0$ steps at height $0$ with the new color $\alpha_0+1$ are those added by $\phi_n$, the inverse map $\phi_n^{-1}$ involves replacing all $D_0$ steps of color $(\alpha_0+1)$ with $U$ steps.
\end{proof}

\begin{figure}[ht!]
\centering
\begin{tikzpicture}
    [scale=.55, auto=left, every node/.style = {circle, fill=black, inner sep=1.25pt}]
    \draw[thick, dotted, color = black] (0,0) to (9,0);
    \draw[thick, dash dot, color = red] (4.5,.5) to (9,.5);
    \draw[thick, dash dot, color = red] (7.5,1.5) to (9,1.5);
    \node(0*) at (0,0) {};
    \node(1*) at (1,1) {};
    \node(2*) at (2,1) {};
    \node(3*) at (3,2) {};
    \node(4*) at (4,0) {};
    \node(5*) at (5,1) {};
    \node(6*) at (6,2) {};
    \node(7*) at (7,1) {};
    \node(8*) at (8,2) {};
    \node(9*) at (9,2) {};
    \draw[thick] (0*) to (1*);
    \draw[thick] (1*) to node[midway,above,fill=none,font=\small] {1} (2*);
    \draw[thick] (2*) to (3*);
    \draw[thick] (3*) to (4*);
    \draw[thick,color=red] (4*) to (5*);
    \draw[thick] (5*) to (6*);
   	\draw[thick] (6*) to node[midway,above,fill=none,font=\small] {1} (7*);
    \draw[thick,color=red] (7*) to (8*);
    \draw[thick] (8*) to node[midway,above,fill=none,font=\small] {1} (9*);
\end{tikzpicture}
\hspace{.3in}
\scalebox{2}{\raisebox{7pt}{$\Leftrightarrow$}}
\hspace{.3in}
\begin{tikzpicture}
    [scale=.55, auto=left, every node/.style = {circle, fill=black, inner sep=1.25pt}]
    \draw[thick, dotted, color = black] (0,0) to (9,0);
    \node(0*) at (0,0) {};
    \node(1*) at (1,1) {};
    \node(2*) at (2,1) {};
    \node(3*) at (3,2) {};
    \node(4*) at (4,0) {};
    \node(5*) at (5,0) {};
    \node(6*) at (6,1) {};
    \node(7*) at (7,0) {};
    \node(8*) at (8,0) {};
    \node(9*) at (9,0) {};
    \draw[thick] (0*) to (1*);
    \draw[thick] (1*) to node[midway,above,fill=none,font=\small] {1} (2*);
    \draw[thick] (2*) to (3*);
    \draw[thick] (3*) to (4*);
    \draw[thick,color=red] (4*) to node[midway,above,fill=none,font=\small] {2} (5*);
    \draw[thick] (5*) to (6*);
   	\draw[thick] (6*) to node[midway,above,fill=none,font=\small] {1} (7*);
    \draw[thick,color=red] (7*) to node[midway,above,fill=none,font=\small] {2} (8*);
    \draw[thick] (8*) to node[midway,above,fill=none,font=\small] {1} (9*);
\end{tikzpicture}
\caption{An example of the bijection from the proof of Theorem \ref{2thm: row-sums of Motzkin triangle, color change version}, here with $\ell=2$, $m=2$, and $\vec{\alpha} = (1,1)$.}
\label{2fig: row-sum bijection example}
\end{figure}

The proof of Theorem \ref{2thm: row-sums of Motzkin triangle, color change version} does not extend to the colorations where $\vec{\alpha} \neq \vec{\beta}$, as the bijection $\phi_n$ may translate $D_i$ steps that end at nonzero height to $D_i$ steps that end at height $0$.  One may still apply Proposition \ref{2thm: height-m Motzkin paths as height-0 convolutions} to Theorem \ref{2thm: row-sums of Motzkin triangle, color change version} to obtain the more general formula below.

\begin{corollary}
\label{2thm: row-sums of Motzkin triangle, color change version corollary}
Fix $\ell \geq 1$ and take any pair of $\ell$-tuples of non-negative integers $\vec{\alpha},\vec{\beta}$.  If $\hat{e}_1 = (1,0,\hdots,0)$, the generating function for $(\vec{\alpha},\vec{\beta})$-colored Motzkin paths of length $n$ and any height $m \geq 0$ is
$$\sum_{m=0}^\infty M_m^\ell(\vec{\alpha},\vec{\beta},t) = M_0^\ell(\vec{\alpha},\vec{\beta},t) + t \kern+2pt M_0^\ell(\vec{\alpha},\vec{\beta},t) \kern+1pt M_0^\ell (\vec{\beta}+ \hat{e}_1,\vec{\beta},t).$$
\end{corollary}
\begin{proof}
Partially-expanding the summation $\sum_{m=0}^\infty M_m^\ell(\vec{\alpha},\vec{\beta},t)$ and applying Proposition \ref{2thm: height-m Motzkin paths as height-0 convolutions} to every $M_m^\ell(\vec{\alpha},\vec{\beta},t)$ with $m \geq 1$ gives
$$ \sum_{m=0}^\infty M_m^\ell(\vec{\alpha},\vec{\beta},t) = M_0^\ell(\vec{\alpha},\vec{\beta},t) + \sum_{m=1}^\infty M_m^\ell(\vec{\alpha},\vec{\beta},t) = M_0^\ell(\vec{\alpha},\vec{\beta},t) + \sum_{m=1}^\infty t \kern+2pt M_0^\ell(\vec{\alpha},\vec{\beta},t) \kern+1pt M_{m-1}^\ell (\vec{\beta},\vec{\beta},t).$$

Rewriting the expression above and applying Theorem \ref{2thm: row-sums of Motzkin triangle, color change version} to the result gives
$$= M_0^\ell(\vec{\alpha},\vec{\beta},t) + t \kern+2pt M_0^\ell(\vec{\alpha},\vec{\beta},t) \sum_{m=0}^\infty M_{m}^\ell (\vec{\beta},\vec{\beta},t) = M_0^\ell(\vec{\alpha},\vec{\beta},t) + t \kern+2pt M_0^\ell(\vec{\alpha},\vec{\beta},t) M_0^\ell (\vec{\beta}+\hat{e}_1,\vec{\beta},t).$$
\end{proof}

We now turn to results involving standard transforms of $(\vec{\alpha},\vec{\beta})$-colored Motzkin numbers.  Given an integer sequence $\lbrace a_i \rbrace_{i=0}^\infty$, recall that the binomial transform of that sequence is the integer sequence $\lbrace b_i \rbrace_{i=0}^\infty$ that satisfies $b_n = \sum_{i=0}^n \binom{n}{i} a_i$ for all $n \geq 0$.  The following theorem characterizes the binomial transform of an arbitrary column in the $(\vec{\alpha},\vec{\beta})$-colored Motzkin triangle.  It should be noted that the $\ell=1,m=0$ case of this theorem, along with the combinatorial interpretation shown in Table \ref{1tab: Catalan-like numbers}, recovers the well-known result that the binomial transform of the Motzkin numbers is the (shifted) Catalan numbers.

\begin{theorem}
\label{2thm: binomial transform of colored Motzkin numbers}
Fix $\ell \geq 1$ and $m \geq 0$, and take any pair of $\ell$-tuples of non-negative integers $\vec{\alpha},\vec{\beta}$.  If $\hat{e}_1 = (1,0,\hdots,0)$, the binomial transform of the sequence $\lbrace M_{n,m}^\ell(\vec{\alpha},\vec{\beta}) \rbrace_{n=0}^\infty$ is the sequence $\lbrace M_{n,m}^\ell(\vec{\alpha} + \hat{e}_1,\vec{\beta} + \hat{e}_1) \rbrace_{n=0}^\infty$.  Explicitly,

$$\sum_{i=0}^n \binom{n}{i} M_{i,m}^\ell(\vec{\alpha},\vec{\beta}) = M_{n,m}^\ell (\vec{\alpha} + \hat{e}_1,\vec{\beta} + \hat{e}_1).$$
\end{theorem}
\begin{proof}
The set $\mathcal{M}_{n,m}^\ell (\vec{\alpha} + \hat{e}_1,\vec{\beta} + \hat{e}_1)$ may be partitioned into the subsets $\bigcup_{i=0}^{n-m} S_i$, where a path $P \in \mathcal{M}_{n,m}^\ell (\vec{\alpha} + \hat{e}_1,\vec{\beta} + \hat{e}_1)$ lies in $S_i$ if and only if it contains precisely $i$ level steps of the final color for its given height ($i$ total $D_0$ steps colored either $\alpha_0+1$ or $\beta_0+1$).  Deleting those $D_0$ steps defines map $\psi_i: S_i \rightarrow \mathcal{M}_{n-i,m}^\ell(\vec{\alpha},\vec{\beta})$ for each $0 \leq i \leq n-m$.  Each map $\psi_i$ is clearly surjective but not injective, with differing locations for the $i$ deleted level steps ensuring that $\binom{n}{i}$ distinct elements of $S_i$ map to each path in $\mathcal{M}_{n-i,m}^\ell(\vec{\alpha},\vec{\beta})$.  It follows that $\vert S_i \vert = \binom{n}{i} M_{n-i,m}^\ell(\vec{\alpha},\vec{\beta})$ for all $0 \leq i \leq n-m$, from which the result follows. 
\end{proof}

\section{Combinatorial Interpretations of $(\vec{\alpha},\vec{\beta})$-Colored\\Motzkin Numbers}
\label{sec: combinatorial interpretations}

For the remainder of this paper, we develop bijections between a variety of well-understood combinatorial objects and collections of $(\vec{\alpha},\vec{\beta})$-colored Motzkin paths.  This will result in a collection of new combinatorial interpretations for $(\vec{\alpha},\vec{\beta})$-colored Motzkin numbers that directly generalize the order $\ell=1$ combinatorial interpretations of Table \ref{1tab: Catalan-like numbers}.

\subsection{$(\vec{\alpha},\vec{\beta})$-Colored Motzkin Numbers and $k$-Dyck Paths}
\label{subsec: k-Dyck paths}

Our first set of combinatorial objects are $k$-Dyck paths, sometimes referred to as $k$-ary paths.  For any $k \geq 2$, a $k$-Dyck path of length $n$ and height $m$ is an integer lattice path from $(0,0)$ to $(n,m)$ that uses the step set $\lbrace U = (1,1), D_{k-1} = (1,1-k) \rbrace$ and remains weakly above $y=0$.  It is obvious that $k$-Dyck paths are in bijection with $(\vec{\alpha},\vec{\beta})$-colored Motzkin paths of order-$(k-1)$ and coloring $\vec{\alpha}=\vec{\beta}=\vec{0}$.  We look for more interesting bijections here.

It can be shown that a point $(x,y)$ may lie on a $k$-Dyck path only if $n = m \kern-4pt \mod \kern-4pt (k)$.  This motivates our choice of dealing only with $k$-Dyck paths of length $kn$ for some $n \geq 0$.  We denote the collection of $k$-Dyck paths of length $kn$ (``semilength" $n$) and height $km$ (``semiheight" $m$) by $\mathcal{D}_{n,m}^k$, and let $\vert \mathcal{D}_{n,m}^k \vert = D_{n,m}^k$.

It is well-known that $k$-Dyck paths of height $0$ are enumerated by the $k$-Catalan numbers (one-parameter Fuss-Catalan numbers) as $D_{n,0}^k = C_n^k = \frac{1}{kn+1}\binom{kn+1}{n}$.  Fixing $k \geq 2$, we define the ordinary generating function $C_k(t) = \sum_{n=0}^\infty C_n^k \kern+1pt t^n$.  It is also well-known that these generating functions satisfy $C_k(t) = 1 + t \kern+1pt C_k(t)^k$ for all $k \geq 2$.  For more information about $k$-Dyck paths and other combinatorial interpretations of the $k$-Catalan numbers, see Hilton and Pedersen \cite{HP} or Heubach, Li and Mansour \cite{HLM}.  

In the order $\ell=1$ case, Table \ref{1tab: Catalan-like numbers} reveals a bijection between $\mathcal{D}_{n,0}^2$ and $(2,2)$-colored Motzkin paths, as well as a bijection between $\mathcal{D}_{n+1,0}^2$ and $(1,2)$-colored Motzkin paths.  These are the bijections that we look to generalize in this subsection.  To do this, begin by observing that $\mathcal{D}_{n+1,0}^2$ is in bijection with ``generalized $2$-Dyck paths" of semilength $n$ that stay weakly above the line $y=-1$, via the map that deletes the initial $U$ step and final $D_1$ step of each $P \in \mathcal{D}_{n+1,0}^2$.

So fix $k \geq 2$ and take any $a \geq 0$.  We define a \textbf{generalized $\mathbf{k}$-Dyck path of depth $\mathbf{a}$}, semilength $n$, and semiheight $m$ to be an integer lattice path from $(0,0)$ to $(kn,km)$ that uses the step set $\lbrace U=(1,1),D_{k-1} = (1,1-k)\rbrace$ and stays weakly above the line $y = -a$.  By analogy with above, we denote the set of all such paths by $\mathcal{D}_{n,m}^{k,a}$ and let $\vert \mathcal{D}_{n,m}^{k,a} \vert = D_{n,m}^{k,a}$.

The sets $\mathcal{D}_{n,m}^{k,a}$ were investigated as ``$k$-Dyck paths with a negative boundary" by Prodinger \cite{Prodinger}, who showed that $D_{n,0}^{k,a} = \frac{a+1}{kn+a+1} \binom{kn+a+1}{n}$ when $0 \leq a \leq k-1$.  Our results will apply over the same range of depths and may be applied to give an alternative derivation of Prodinger's closed formula in terms of proper Riordan arrays.

Now fix $k \geq 2$ and $a \geq 0$, and define $D^{k,a}$ to be the infinite, lower-triangular array of non-negative integers whose $(n,m)$ entry (for $0 \leq m \leq n$) is $D_{n,m}^{k,a}$.  Our approach is to show that these integer triangles represent the same proper Riordan arrays as $(\vec{\alpha},\vec{\beta})$-colored Motzkin triangles for particular choices of $\vec{\alpha},\vec{\beta}$.

\begin{theorem}
\label{3thm: Dyck triangle, A and Z sequences}
For any $k \geq 2$ and $0 \leq a \leq k-1$, $D^{k,a}$ is a proper Riordan array with $A$- and $Z$-sequences
$$A(t) = (1+t)^k, \hspace{.5in} Z(t) = \frac{(1+t)^k - (1+t)^{k-a-1}}{t}.$$
\end{theorem}
\begin{proof}
As in the proof of Proposition \ref{2thm: higher-order Motzkin triangle as Riordan array}, it suffices to prove the recurrences

\vspace{.1in}

$D_{n,m}^{k,a} = 
\begin{cases}
\binom{k}{0} D_{n-1,m-1}^{k,a} + \hdots + \binom{k}{k} D_{n-1,m+k-1}^{k,a} & \text{if } m \geq 1; \\[6pt]
\left( \binom{k}{1} - \binom{k-a-1}{1} \right) D_{n-1,0}^{k,a} + \hdots + \left( \binom{k}{k} - \binom{k-a-1}{k} \right) D_{n-1,k-1}^{k,a} & \text{if } m =0.
\end{cases}$

\vspace{.1in}

For any $m \geq 0$, we partition $\mathcal{D}_{n,m}^{k,a}$ into sets $S_{Q_2}$ whereby $P \in \mathcal{D}_{n,m}^{k,a}$ lies in $S_{Q_2}$ if $P$ decomposes as $P = Q_1 Q_2$ for the length-$k$ terminal subpath $Q_2$.  For $P = Q_1 Q_2$, observe that $Q_1 \in \mathcal{D}_{n-1,m-1+j}^{k,a}$ if $Q_2$ contains precisely $j$ down steps.  This implies that $\vert S_{Q_2} \vert = D_{n-1,m-1+j}^{k,a}$ for every valid choice of $Q_2$, via the bijection that takes $P = Q_1 Q_2$ to $Q_1$.

All that's left is to enumerate length-$k$ subpaths $Q_2$ that end with semiheight $m \geq 0$, contain precisely $0 \leq j \leq k$ down steps, and remain weakly above $y=-a$.  When $m \geq 1$, it is impossible for such a subpath (for any $j$) to go below $y=0$ and still end at height $(k-1)m$.  This would require a full complement of $k$ up steps to travel from $y=-1$ to a height of at least $(k-1)$, and we're assuming that $Q_2$ began with non-negative height.  It follows that there are $\binom{k}{j}$ valid choices of $Q_2$ with precisely $j$ down steps when $m \geq 1$, giving the first line of our desired recurrence.

When $m = 0$, not all $\binom{k}{j}$ potential subpaths $Q_2$ will remain weakly above $y=-a$.  We enumerate the ``bad" length-$k$ subpaths that that go below $y=-a$.  Every such ``bad" subpath $Q_2$ has a rightmost step $p$ that terminates at height $-a-1$.  We claim that $p$ may only be followed by up steps.  This is because, if $p$ were followed by any $D_{k-1}$ steps, then $p$ would also need to be followed by at least $(k-1)+(a+1) \geq k$ up steps if we want $Q_2$ to end at height $0$.  It follows that ``bad" subpaths $Q_2$ must decompose as $Q_2 = Q_3 \kern+1pt p \kern+1pt U^{a+1}$, with the $j$ down steps of $Q_2$ being distributed among the $k-a-1$ steps of the sub-subpath $Q_3 \kern+1pt p$.  It follows that there are precisely $\binom{k-a-1}{j}$ ``bad" choices of $Q_2$ with precisely $j$ down steps, and thus that there are precisely $\binom{k}{j} - \binom{k-a-1}{j}$ valid choices of $Q_2$ with precisely $j$ down steps.  This gives the second line of our desired recurrence.
\end{proof}

\begin{corollary}
\label{3thm: Dyck paths vs colored Motzkin paths}
Fix $k \geq 2$.  For all $n \geq 0$, $0 \leq m \leq n$, and $0 \leq a \leq k-1$, the equality $D_{n,m}^{k,a} = M_{n,m}^{k-1}(\vec{\alpha},\vec{\beta})$ holds for $(k-1)$-tuples $\vec{\alpha} = (\alpha_0,\hdots,\alpha_{k-2})$ and $\vec{\beta}=(\beta_0,\hdots,\beta_{k-2})$ with $\alpha_i = \binom{k}{i+1} - \binom{k-a-1}{i+1}$ and $\beta_i = \binom{k}{i+1}$ for all $0 \leq i \leq k-2$.
\end{corollary}
\begin{proof}
This follows directly from a comparison of Proposition \ref{2thm: higher-order Motzkin triangle as Riordan array} and Theorem \ref{3thm: Dyck triangle, A and Z sequences}.
\end{proof}

Note the order shift of Corollary \ref{3thm: Dyck paths vs colored Motzkin paths}: generalized $k$-Dyck paths correspond to order-$(k-1)$ colored Motzkin paths.  Also notice this represents a bijection between generalized Dyck paths of length $kn$ and colored Motzkin paths of length $n$.  The reason that Theorem \ref{3thm: Dyck triangle, A and Z sequences} and Corollary \ref{3thm: Dyck paths vs colored Motzkin paths} fail to generalize to $a \geq k$ is because generalized Dyck paths of those depths may have negative semidepth $m$, making it is impossible to arrange the cardinalities $D_{n,m}^{k,a}$ into a proper Riordan array.

For $k=2,3,4$, the $(k-1)$-tuples $(\vec{\alpha},\vec{\beta})$ that result from Corollary \ref{3thm: Dyck paths vs colored Motzkin paths} are summarized in Table \ref{3tab: k-Dyck paths}.  See Appendix \ref{sec: appendix} for how these colorations fit within the broader scheme of $(\vec{\alpha},\vec{\beta})$-colored Motzkin paths.

\begin{table}[ht!]
\centering
\begin{tabular}{|c|c|c|c|c|}
\hline
$\vec{\alpha},\vec{\beta}$ & $\boldsymbol{a=0}$ & $\boldsymbol{a=1}$ & $\boldsymbol{a=2}$ & $\boldsymbol{a=3}$ \\ \hline
$\boldsymbol{k=2}$ & $(1),(2)$ & $(2),(2)$ & - & -\\ \hline
$\boldsymbol{k=3}$ & $(1,2), (3,3)$ & $(2,3), (3,3)$ & $(3,3), (3,3)$ & -\\ \hline
$\boldsymbol{k=4}$ & $(1,3,3),(4,6,4)$ & $(2,5,4),(4,6,4)$ & $(3,6,4),(4,6,4)$ & $(4,6,4),(4,6,4)$\\ \hline
\end{tabular}
\caption{$(k-1)$-tuples $\vec{\alpha},\vec{\beta}$ such that $D_{n,m}^{k,a} = M_{n,m}^{k-1}(\vec{\alpha},\vec{\beta})$, as proven in Corollary \ref{3thm: Dyck paths vs colored Motzkin paths}.}
\label{3tab: k-Dyck paths}
\end{table}

\vspace{.1in}

As it will be helpful in upcoming subsections, we briefly outline one explicit bijection between $\mathcal{D}_{n,m}^{k,a}$ and $\mathcal{M}_{n,m}^{k-1}(\vec{\alpha},\vec{\beta})$ for the tuples $\vec{\alpha},\vec{\beta}$ of Corollary \ref{3thm: Dyck paths vs colored Motzkin paths}.  So take $P \in \mathcal{M}_{n,m}^{k-1}(\vec{\alpha},\vec{\beta})$, and take any step $p$ of $P$ that begins at height $y_1$ and ends at height $y_2$.  We replace $p$ with a length-$k$ subpath with step set $\lbrace U,D_{k-1} \rbrace$ that begins at height $k y_1$ and ends at height $k y_2$.  If $p$ is a $U$ step, the only way to do this is with a subpath of $k$ consecutive $U$ steps.  If $p$ is a $D_i$ step, the new subpath must contain precisely $i+1$ total $D_{k-1}$ steps and $k-i-1$ total $U$ steps.  There are $\binom{k}{i+1}$ such length-$k$ subpaths: the specific subpath chosen is determined by the coloring of the $D_i$ step being replaced.  All such length-$k$ subpaths stay above $y=0$ (and hence above $y=-a$) when $y_2 \geq 1$, whereas some subpaths may go below $y=-a$ when $y_2 = 0$.  The colorations $(\vec{\alpha},\vec{\beta})$ of Corollary \ref{3thm: Dyck paths vs colored Motzkin paths} provide the number of valid length-$k$ subpaths.  See Figure \ref{3fig: Motzkin to k-Dyck bijection example} for an example of this bijection.

\begin{figure}[ht!]
\centering
\begin{tikzpicture}
    [scale=.55, auto=left, every node/.style = {circle, fill=black, inner sep=1.25pt}]
    \draw[thick, dotted, color = black] (0,0) to (6,0);
    \node(0*) at (0,0) {};
    \node(1*) at (1,0) {};
    \node(2*) at (2,1) {};
    \node(3*) at (3,1) {};
    \node(4*) at (4,1) {};
    \node(5*) at (5,0) {};
    \node(6*) at (6,0) {};
    \draw[thick] (0*) to node[midway,above,fill=none,font=\small] {1} (1*);
    \draw[thick] (1*) to (2*);
    \draw[thick] (2*) to node[midway,above,fill=none,font=\small] {2} (3*);
    \draw[thick] (3*) to node[midway,above,fill=none,font=\small] {1} (4*);
    \draw[thick] (4*) to (5*);
    \draw[thick] (5*) to node[midway,above,fill=none,font=\small] {2} (6*);
\end{tikzpicture}
\hspace{.3in}
\scalebox{2.5}{\raisebox{4pt}{$\Leftrightarrow$}}
\hspace{.3in}
\raisebox{-10pt}{
\begin{tikzpicture}
    [scale=.4, auto=left, every node/.style = {circle, fill=black, inner sep=1.1pt}]
    \draw[thick, dotted, color = black] (0,0) to (12,0);
    \node(0*) at (0,0) {};
    \node(1*) at (1,1) {};
    \node(2*) at (2,0) {};
    \node(3*) at (3,1) {};
    \node(4*) at (4,2) {};
    \node(5*) at (5,1) {};
    \node(6*) at (6,2) {};
    \node(7*) at (7,3) {};
    \node(8*) at (8,2) {};
    \node(9*) at (9,1) {};
    \node(10*) at (10,0) {};
    \node(11*) at (11,-1) {};
    \node(12*) at (12,0) {};
    \draw[thick] (0*) to (1*);
    \draw[thick] (1*) to (2*);
    \draw[thick] (2*) to (3*);
    \draw[thick] (3*) to (4*);
    \draw[thick] (4*) to (5*);
    \draw[thick] (5*) to (6*);
   	\draw[thick] (6*) to (7*);
    \draw[thick] (7*) to (8*);
    \draw[thick] (8*) to (9*);
    \draw[thick] (9*) to (10*);
    \draw[thick] (10*) to (11*);
    \draw[thick] (11*) to (12*);
\end{tikzpicture}}
\caption{An example of our bijection between $\mathcal{M}_{n,m}^{k-1}(\vec{\alpha},\vec{\beta})$ and $\mathcal{D}_{n,m}^{k,a}$ for $k=2$ and $a=1$.  Here, $U$ steps in the Motzkin path are replaced by a $UU$ subpath in the $2$-Dyck , $D_1$ steps are replaced by $DD$, $D_0$ steps of color $1$ are replaced by $UD$, and $D_0$ steps of color $2$ are replaced by $DU$.  Notice that forbidding $D_0$ steps of color $2$ at height $y=0$ prevents the resulting $2$-Dyck path from dropping below $y=0$.}
\label{3fig: Motzkin to k-Dyck bijection example}
\end{figure}

Theorem \ref{3thm: Dyck triangle, A and Z sequences} may be used to find the generating functions $d(t),h(t)$ of the proper Riordan array $\mathcal{R}(d(t),h(t))$ with entries $D_{n,m}^{k,a}$.  As seen in Corollary \ref{3thm: Dyck path triangles as Fuss-Catalan triangles}, these Riordan arrays are ``Fuss-Catalan triangles" of the type introduced by He and Shapiro \cite{HS} and further examined by Drube \cite{Drube}.

\begin{corollary}
\label{3thm: Dyck path triangles as Fuss-Catalan triangles}
For any $k \geq 2$ and $0 \leq a \leq k-1$, $D^{k,a}$ is the proper Riordan array $\mathcal{R}(d(t),h(t))$ with $d(t) = C_k(t)^{a+1}$ and $h(t) = t \kern+1pt C_k(t)^k$.
\end{corollary}
\begin{proof}
Given $A(t)$ and $Z(t)$ from Theorem \ref{3thm: Dyck triangle, A and Z sequences}, we merely need to verify the identities of \eqref{1eq: A and Z sequences vs d(t),h(t)}.  Using the $k$-Catalan identity $C_k(t) = 1 + t \kern+1pt C_k(t)^k$, we have

$$t \kern+1pt A(h(t)) = t \kern+1pt (1 + t \kern+1pt C_k(t)^k)^k = t \kern+1pt C_k(t)^k = h(t), \text{ and}$$

$$\frac{d(0)}{1 - t \kern+1pt Z(h(t))} = \frac{1}{1 - t \left( \frac{(1+t \kern+1pt C_k(t)^k)^k - (1 + t \kern+1pt C_k(t)^k)^{k-a-1}}{t \kern+1pt C_k(t)^k}\right)} = \frac{1}{1- \frac{C_k(t)^k - C_k(t)^{k-a-1}}{C_k(t)^k}}$$
$$= \frac{1}{\frac{C_k(t)^k - C_k(t)^k + C_k(t)^{k-a-1}}{C_k(t)^k}} = \frac{C_k(t)^k}{C_k(t)^{k-a-1}} = C_k(t)^{a+1} = d(t).$$
\end{proof}

For one final result involving $k$-Dyck paths, notice that the $a=k-1$ case of Corollary \ref{3thm: Dyck paths vs colored Motzkin paths} places $\mathcal{D}_{n,m}^{k,k-1}$ in bijection with $\mathcal{M}_{n,m}^{k-1}(\vec{\alpha},\vec{\beta})$ for which $\vec{\alpha}=\vec{\beta}$.  We may then apply Theorem \ref{2thm: row-sums of Motzkin triangle, color change version} to enumerate generalized $k$-Dyck paths of fixed length and any semiheight.

\begin{corollary}
\label{3thm: Dyck paths row sums}
Fix $k \geq 2$.  For all $n \geq 0$, the equality $\sum_{m=0}^n D_{n,m}^{k,k-1} = M_{n,m}^{k-1}(\vec{\alpha},\vec{\alpha})$ holds for the $(k-1)$-tuple $\vec{\alpha} = (\alpha_0,\hdots,\alpha_{k-2})$ with $\alpha_0 = k+1$ and $\alpha_i = \binom{k}{i+1}$ for all $1 \leq i \leq k-2$. 
\end{corollary}

\subsection{$(\vec{\alpha},\vec{\beta})$-Colored Motzkin Numbers and $(k,r)$-Fine Paths}
\label{subsec: k-Fine paths}

Our second set of combinatorial objects are subsets of $k$-Dyck paths that we refer to as $(k,r)$-Fine paths.  These are an entirely new notion that intuitively generalize the concept of Fine paths from $k=2$ to all $k \geq 2$.

A Fine path of length $n$ and height $m$ is an element of $D_{n,m}^2$ that lacks a subpath of the form $U D_1$ ending at height $y=0$.  Forbidden subpaths of this type are called ``hills", meaning that Fine path are $2$-Dyck paths that lack hills.  It is well-known that Fine paths of height $0$ are enumerated by the Fine numbers $\lbrace F_n \rbrace_{n=0}^\infty$, an integer sequence that begins $1,0,1,2,6,18,\hdots$.  The Fine numbers have an ordinary generating function $F(t)$ that satisfies $F(t) = \frac{1}{1 - t^2 C_2(t)^2} = \frac{C_2(t)}{1+ t \kern+1pt C_2(t)}$.  For more results about Fine paths and the Fine numbers, see Deutsch and Shapiro \cite{DS}.

The notion of a ``hill" becomes more ambiguous when you generalize from $2$-Dyck paths to $k$-Dyck paths when $k > 2$.  For fixed $k \geq 2$, we identify $k-1$ competing definitions for a Fuss-Catalan analogue of Fine paths, each of which forbids different classes of subpaths that end at height $0$.  So fix $k \geq 2$, and take any $1 \leq r \leq k-1$.  A \textbf{$\mathbf{(k,r)}$-Fine path} of semilength $n$ and semiheight $m$ is an element of $\mathcal{D}_{n,m}^k$ that lacks a subpath of the form $U^r D_{k-1}$ that ends at height $0$.  Clearly, $(k,r_1)$-Fine paths are a subset of $(k,r_2)$-Fine paths for all $r_1 < r_2$.  See Figure \ref{3fig: generalized Fine path example} for a simple example.

\begin{figure}[ht!]
\centering
\begin{tikzpicture}
	[scale=0.45, auto=left, every node/.style = {circle, fill=black, inner sep=1.25pt}]
	\node(0*) at (0,0) {};
	\node(1*) at (.8,1) {};
	\node(2*) at (1.6,2) {};
	\node(3*) at (2.4,3) {};
	\node(4*) at (3.2,4) {};
	\node(5*) at (4,2) {};
	\node(6*) at (4.8,0) {};
	\draw[dotted, color = black] (0*) to (6*);
	\draw[thick,black] (0*) to (1*);
	\draw[thick,black] (1*) to (2*);
	\draw[thick,black] (2*) to (3*);
	\draw[thick,black] (3*) to (4*);
	\draw[thick,black] (4*) to (5*);
	\draw[thick,black] (5*) to (6*);
\end{tikzpicture}
\hspace{.5in}
\begin{tikzpicture}
	[scale=0.45, auto=left, every node/.style = {circle, fill=black, inner sep=1.25pt}]
	\node(0*) at (0,0) {};
	\node(1*) at (.8,1) {};
	\node(2*) at (1.6,2) {};
	\node(3*) at (2.4,3) {};
	\node(4*) at (3.2,1) {};
	\node(5*)[color=red] at (4,2) {};
	\node(6*)[color=red] at (4.8,0) {};
	\draw[dotted, color = black] (0*) to (6*);
	\draw[thick,black] (0*) to (1*);
	\draw[thick,black] (1*) to (2*);
	\draw[thick,black] (2*) to (3*);
	\draw[thick,black] (3*) to (4*);
	\draw[thick,red] (4*) to (5*);
	\draw[thick,red] (5*) to (6*);
\end{tikzpicture}
\hspace{0.5in}
\begin{tikzpicture}
	[scale=0.45, auto=left, every node/.style = {circle, fill=red, inner sep=1.25pt}]
	\node(0*) at (0,0) {};
	\node(1*) at (.8,1) {};
	\node(2*) at (1.6,2) {};
	\node(3*) at (2.4,0) {};
	\node(4*) at (3.2,1) {};
	\node(5*) at (4,2) {};
	\node(6*) at (4.8,0) {};
	\draw[dotted, color = black] (0*) to (6*);
	\draw[thick,red] (0*) to (1*);
	\draw[thick,red] (1*) to (2*);
	\draw[thick,red] (2*) to (3*);
	\draw[thick,red] (3*) to (4*);
	\draw[thick,red] (4*) to (5*);
	\draw[thick,red] (5*) to (6*);
\end{tikzpicture}
\caption{A trio of $3$-Dyck paths, the first of which is both $(3,2)$-Fine and $(3,1)$-Fine, the second of which is $(3,2)$-Fine but not $(3,1)$-Fine, and the third of which is neither $(3,2)$-Fine nor $(3,1)$-Fine.  Generalized ``hills" are shown in red.}
\label{3fig: generalized Fine path example}
\end{figure}
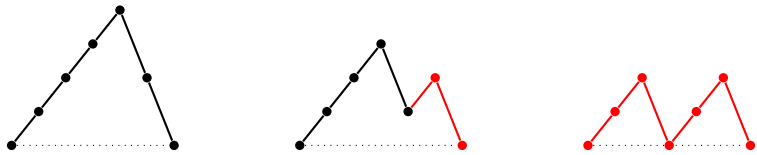

As in Subsection \ref{subsec: k-Dyck paths}, we further generalize the notion of $(k,r)$-Fine paths to paths that stay weakly above $y = -a$ for any $0 \leq a \leq k-1$.  We define a \textbf{generalized $\mathbf{(k,r)}$-Fine path of depth $\textbf{a}$}, semilength $n$, and semiheight $m$ to be an element of $\mathcal{D}_{n,m}^{k,a}$ lacking a subpath of the form $U^r D_{k-1}$ that ends at height $0$.  We denote the set of all such paths $\mathcal{F}_{n,m}^{k,a,r}$ and let $\vert \mathcal{F}_{n,m}^{k,a,r} \vert = F_{n,m}^{k,a,r}$.  We refer to the sequences $\lbrace F_{n,0}^{k,0,r} \rbrace_{n=0}^\infty$ as the \textbf{$\mathbf{(k,r)}$-Fine numbers}.

Also mirroring Subsection \ref{subsec: k-Dyck paths}, for any $k \geq 2$, $a \geq 0$, and $1 \leq r \leq k-1$ we define $F^{k,a,r}$ to be the infinite, lower-triangular array whose $(n,m)$ entry (for $0 \leq m \leq n)$ is $F_{n,m}^{k,a,r}$.  Our approach is once again to identify the proper Riordan array associated with each triangle $F^{k,a,r}$ and compare the results to Proposition \ref{2thm: higher-order Motzkin triangle as Riordan array}.

\begin{theorem}
\label{3thm: Fine triangle, A and Z sequences}
For any $k \geq 2$, $0 \leq a \leq k-1$, and $1 \leq r \leq k-1$, $F^{k,a,r}$ is a proper Riordan array with $A$- and $Z$-sequences
$$A(t) = (1+t)^k, \hspace{.5in} Z(t) = \frac{(1+t)^k - (1+t)^{k-a-1}}{t} - (1+t)^{k-r-1}.$$
\end{theorem}
\begin{proof}
The argument is largely equivalent to the proof of Theorem \ref{3thm: Dyck triangle, A and Z sequences}.  The only difference comes in the $m=0$ case, where we also need to exclude potential length-$k$ terminal subpaths $Q_2$ that introduce a ``hill" of the form $U^r D_{k-1}$.

As a point along a $k$-Dyck path can only return to $y=0$ when its $x$-coordinate is divisible by $k$, all subpaths $Q_2$ that introduce a hill $U^r D_{k-1}$ must do so over their final $r+1$ steps.  This leaves $k-r-1$ steps at the beginning of our terminal subpath, the totality of which must begin at height $kj$ (for some $j \geq 0$) and end at height $y=k-1-r$.  It follows that, if $Q_2$ begins at height $kj$ and ends at height $0$, it must contain precisely $j+1$ steps of type $D_{k-1}$, and that precisely $j$ of those $D_{k-1}$ steps must be within its first $k-r-1$ steps. All of this means there are precisely $\binom{k-r-1}{j}$ potential subpaths $Q_2$ that begin at height $kj$, end at height $0$, and introduce a hill of the form $U^r D_{k-1}$.

Since they must reach a height of $k-1-r \geq 0$ after their first $k-r-1$ steps, the aforementioned restrictions of when a $k$-Dyck path can return to $y=0$ ensures that none of the $\binom{k-r-1}{j}$ hill-introducing subpaths $Q_2$ enumerated above can go below $y=0$.  This ensures that no potential subpaths $Q_2$ are ``doubly excluded" when citing the proof of Theorem \ref{3thm: Dyck triangle, A and Z sequences}, and we may modify the $m=0$ recurrence from that theorem to give

\begin{center}
$F_{n,0}^{k,a,r} = \left( \binom{k}{1} - \binom{k-a-1}{1} - \binom{k-r-1}{0} \right) F_{n-1,0}^{k,a,r} + \hdots + \left( \binom{k}{k} - \binom{k-a-1}{k} - \binom{k-r-1}{k-1} \right) F_{n-1,k-1}^{k,a,r}$
\end{center}

The desired $A$-sequence carries over from Theorem \ref{3thm: Dyck triangle, A and Z sequences}, whereas the desired $Z$-sequence follows directly from the recurrence above.
\end{proof}

\begin{corollary}
\label{3thm: Fine paths vs colored Motzkin paths}
Fix $k \geq 2$.  For all $n \geq 0$, $0 \leq m \leq n$, $0 \leq a \leq k-1$, and $1 \leq r \leq k-1$, the equality $F_{n,m}^{k,a,r} = M_{n,m}^{k-1}(\vec{\alpha},\vec{\beta})$ holds for the $(k-1)$-tuples $\vec{\alpha} = (\alpha_0,\hdots,\alpha_{k-2})$ and $\vec{\beta} = (\beta_0,\hdots,\beta_{k-2})$ with $\alpha_i = \binom{k}{i+1} - \binom{k-a-1}{i+1} - \binom{k-r-1}{i}$ and $\beta_i = \binom{k}{i+1}$ for all $0 \leq i \leq k-2$.
\end{corollary}
\begin{proof}
Follows from a comparison of Proposition \ref{2thm: higher-order Motzkin triangle as Riordan array} and Theorem \ref{3thm: Fine triangle, A and Z sequences}
\end{proof}

For $k=2,3,4$ and $a=0$, the $(k-1)$-tuples $(\vec{\alpha},\vec{\beta})$ guaranteed by Corollary \ref{3thm: Fine paths vs colored Motzkin paths} are summarized in Table \ref{3tab: k-Fine paths}.  Once again, see Appendix \ref{sec: appendix} for how these colorations fit within the wider context of $(\vec{\alpha},\vec{\beta})$-colored Motzkin paths.

\begin{table}[ht!]
\centering
\begin{tabular}{|c|c|c|c|}
\hline
$\vec{\alpha},\vec{\beta}$ & $\boldsymbol{a=0,r=1}$ & $\boldsymbol{a=0,r=2}$ & $\boldsymbol{a=0,r=3}$ \\ \hline
$\boldsymbol{k=2}$ & $(0),(2)$ & - & -\\ \hline
$\boldsymbol{k=3}$ & $(0,1),(3,3)$ & $(0,2),(3,3)$ & -\\ \hline
$\boldsymbol{k=4}$ & $(0,1,2),(4,6,4)$ & $(0,2,3),(4,6,4)$ & $(0,3,3),(4,6,4)$\\ \hline
\end{tabular}
\caption{$(k-1)$-tuples $\vec{\alpha},\vec{\beta}$ such that $F_{n,0}^{k,a,r} = M_{n,0}^{k-1}(\vec{\alpha},\vec{\beta})$, as proven in Corollary \ref{3thm: Fine paths vs colored Motzkin paths}.}
\label{3tab: k-Fine paths}
\end{table}

The explicit bijection between $\mathcal{D}_{n,m}^{k,a}$ and $\mathcal{M}_{n,m}^{k-1}(\vec{\alpha},\vec{\beta})$ from Subsection \ref{subsec: k-Dyck paths} restricts to a bijection between $\mathcal{F}_{n,m}^{k,a,r}$ and a (distinct) set of colored Motzkin paths $\mathcal{M}_{n,0}^{k-1}(\vec{\alpha},\vec{\beta})$, with $\vec{\alpha},\vec{\beta}$ as determined by Corollary \ref{3thm: Fine paths vs colored Motzkin paths}.  Also similar to Subsection \ref{subsec: k-Dyck paths} is the fact Theorem \ref{3thm: Fine triangle, A and Z sequences} may be used to characterize the integer triangles $F^{k,a,r}$ as proper Riordan arrays:

\begin{corollary}
\label{3thm: Fine triangles as Fuss-Catalan triangles}
For any $k \geq 2$, $0 \leq a \leq k-1$, and $1 \leq r \leq k-1$, $F^{k,a,r}$ is the proper Riordan array $\mathcal{R}(d(t),h(t))$ with $d(t) = \frac{C_k(t)^k}{C_k(t)^{k-a-1} + t \kern+1pt C_k(t)^{2k-r-1}}$ and $h(t) = t \kern+1pt C_k(t)^k$.
\end{corollary}
\begin{proof}
We use $A(t), Z(t)$ from Theorem \ref{3thm: Fine triangle, A and Z sequences} to verify the identities of \eqref{1eq: A and Z sequences vs d(t),h(t)}.  As $A(t)$ is the same as in Subsection \ref{subsec: k-Dyck paths}, verification of $h(t)$ is identical to the proof of Corollary \ref{3thm: Dyck path triangles as Fuss-Catalan triangles}.  Verification of $d(t)$ now takes the form.

$$\frac{d(0)}{1 - t \kern+1pt Z(h(t))} = \frac{1}{1 - \left( \frac{(1 + t \kern+1pt C_k(t)^k)^k - (1 + t \kern+1pt C_k(t)^k)^{k-a-1}}{t \kern+1pt C_k(t)^k} - (1 + t \kern+1pt C_k(t)^k)^{k-r-1} \right)}$$
$$= \frac{1}{1 - t \left( \frac{C_k(t)^k - C_k(t)^{k-a-1}}{t \kern+1pt C_k(t)^k} - C_k(t)^{k-r-1} \right)} = \frac{1}{1 - 1 + \frac{C_k(t)^{k-a-1}}{C_k(t)^k} + t \kern+1pt C_k(t)^{k-r-1}}$$
$$\frac{C_k(t)^k}{C_k(t)^{k-a-1} + t \kern+1pt C_k(t)^{2k-r-1}} = d(t).$$
\end{proof}

Substituting $a=0$ into $d(t)$ from Corollary \ref{3thm: Fine triangles as Fuss-Catalan triangles} provides a relatively simple relationship for the generating function $F_{k,r}(t)$ of the $(k,r)$-Fine numbers $\lbrace F_{n,0}^{k,0,r} \rbrace_{n=0}^\infty$ as $F_{k,r}(t) = \frac{C_k(t)}{1 + t \kern+1pt C_k(t)^{k-r}}$.  Observe that this formula simplifies to the well-known relationship of $F(t) = \frac{C_2(t)}{1 + t \kern+1pt C_2(t)}$ in the case of $k=2,r=1$.  

\subsection{$(\vec{\alpha},\vec{\beta})$-Colored Motzkin Numbers and $k$-Dyck Paths\\with Restrictions on Peak Heights}
\label{subsec: peak-restricted k-Dyck paths}

We now consider subsets of $k$-Dyck paths whose peaks must appear at a fixed height, modulo $k$.  By a ``peak" we mean any subpath of the form $U D_{k-1}$, with the height of a peak equaling the height of (the right end of) the $U$ step in the subpath.

For any $0 \leq i \leq k-1$, we say that $P \in \mathcal{D}_{n,m}^{k,a}$ has \textbf{peak parity} $i$ if the height of every one of its peaks is equivalent to $i \kern-6pt \mod \kern-3pt (k)$.  We denote the subset of $\mathcal{D}_{n,m}^{k,a}$ consisting of all paths with peak parity $i$ by $\mathcal{D}_{n,m}^{k,a}(i)$.  To avoid the ambiguity of categorizing paths of length $0$, which have no peaks, we henceforth restrict our attention to sets $\mathcal{D}_{n,m}^{k,a}(i)$ with $n > 0$.

The goal of this subsection is to generalize the following pair of identities, which were originally proven by Callan \cite{Callan}:

\begin{enumerate}
\item The set $\mathcal{D}_{n,0}^{2,0}(0)$ consisting of all $2$-Dyck paths with peaks only at even height is enumerated by the Riordan numbers $R_n = M_{n,0}(0,1)$. 
\item The set $\mathcal{D}_{n,0}^{2,0}(1)$ consisting of all $2$-Dyck paths with peaks only at odd height is enumerated by the shifted Motzkin numbers $M_{n-1} = M_{n-1,0}(1,1)$.
\end{enumerate}

\noindent The parity-$0$ case, corresponding to Callan's first identity, may be directly generalized as follows:

\begin{theorem}
\label{3thm: peak parity 0}
Fix $k \geq 2$ and $0 \leq a \leq k-1$.  Then $\vert \mathcal{D}_{n,m}^{k,a}(0) \vert = M_{n,m}^{k-1}(\vec{\alpha},\vec{1})$ for all $n,m$, where $\vec{1} = (1,1,\hdots,1)$ and $\vec{\alpha} = (\alpha_0,\hdots,\alpha_{k-2})$ satisfies
$$\alpha_i=
\begin{cases}
1    &\text{if }i<a; \\
0    &\text{if }i \geq a.
\end{cases}$$
\end{theorem}
\begin{proof}
As points $(x,y)$ along $P \in \mathcal{D}_{n,m}^{k,a}(0)$ must satisfy $x = y \kern-5pt \mod \kern-3pt (k)$, all peaks in $P \in \mathcal{D}_{n,m}^{k,a}(0)$ must have $x$-coordinates that are divisible by $k$.  This means that every such $P$ may be subdivided into a sequence of length-$k$ subpaths, each of which is of the form $(D_{k-1})^i U^{k-i}$ for some $0 \leq i \leq k$.  For the subpath $(D_{k-1})^i U^{k-i}$, the subheight of the terminal point minus the subheight of the initial point is $1-i$.  Also notice that the lowest point in the subpath $(D_{k-1})^i U^{k-i}$ is $k-i$ units lower than the terminal point of the subpath.

We define an explicit bijection $\psi: \mathcal{M}_{n,m}^{k-1}(\vec{\alpha},\vec{1}) \rightarrow \mathcal{D}_{n,m}^{k,a}(0)$ that is similar to the bijection described in Subsection \ref{subsec: k-Dyck paths}.  So take any $P \in \mathcal{M}_{n,m}^{k-1}(\vec{\alpha},\vec{1})$.  To obtain $\psi(P)$, replace each $U$ step of $P$ with the length-$k$ subpath $U^k$, and (for each $0 \leq i \leq k-1$) replace each $D_i$ step of $P$ with the length-$k$ subpath $(D_{k-1})^{i+1} U^{k-i-1}$.  For an example of this map, see Figure \ref{3fig: peak parity lemma example}.

The map $\psi$ is clearly injective, and its image is clearly some subset of generalized $k$-Dyck path of semilength $n$ and semiheight $m$.  By construction, all peaks in $\psi(P)$ are at a height of $0 \kern-6pt \mod \kern-3pt (k)$.  As the steepest allowable down steps of $P$ ending at height $0$ are $D_{a-1}$, the ``lowest-dipping" length-$k$ subpaths of $\psi(P)$ that end at height $0$ are $(D_{k-1})^a U^{k-a}$ and they reach a minimum height of $y=-a$.  Thus $\psi(P)$ remains weakly above $y=-a$.  The fact that $\psi$ is surjective follows from the observation that elements of $\mathcal{D}_{n,m}^{k,a}(0)$ have one allowable length-$k$ subpath with a particular overall change in subheight.  The choice of $\vec{\alpha}$ restricts the options for subpaths that end at subheight $0$ to those that stay weakly above $y=-a$. 
\end{proof}

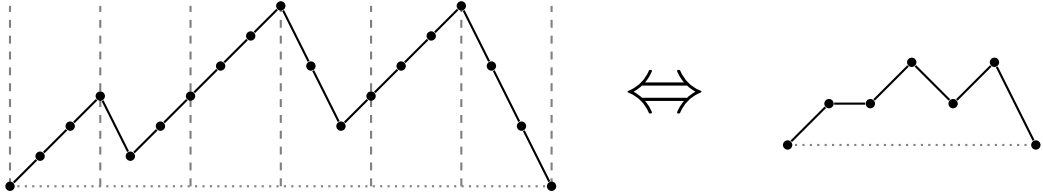
\begin{figure}[ht!]
\centering
\begin{tikzpicture}
    [scale=.4, auto=left, every node/.style = {circle, fill=black, inner sep=1.25pt}]
    \draw[thick,dotted, color = gray] (0,0) to (18,0);
    \draw[thick,dashed, color = gray] (0,0) to (0,6);
    \draw[thick,dashed, color = gray] (3,0) to (3,6);
    \draw[thick,dashed, color = gray] (6,0) to (6,6);
    \draw[thick,dashed, color = gray] (9,0) to (9,6);
    \draw[thick,dashed, color = gray] (12,0) to (12,6);
    \draw[thick,dashed, color = gray] (15,0) to (15,6);
    \draw[thick,dashed, color = gray] (18,0) to (18,6);    
    \node(1*) at (0,0) {};
    \node(2*) at (1,1) {};
    \node(3*) at (2,2) {};
    \node(4*) at (3,3) {};
    \node(5*) at (4,1) {};
    \node(6*) at (5,2) {};
    \node(7*) at (6,3) {};
    \node(8*) at (7,4) {};
    \node(9*) at (8,5) {};
    \node(10*) at (9,6) {};
    \node(11*) at (10,4) {};
    \node(12*) at (11,2) {};
    \node(13*) at (12,3){};
    \node(14*) at (13,4) {};
    \node(15*) at (14,5) {};
    \node(16*) at (15,6) {};
    \node(17*) at (16,4) {};
    \node(18*) at (17,2) {};
    \node(19*) at (18,0) {};
    \draw[thick] (1*) to (2*);
    \draw[thick] (2*) to (3*);
    \draw[thick] (3*) to (4*);
    \draw[thick] (4*) to (5*);
    \draw[thick] (5*) to (6*);
    \draw[thick] (6*) to (7*);
    \draw[thick] (7*) to (8*);
    \draw[thick] (8*) to (9*);
    \draw[thick] (9*) to (10*);
    \draw[thick] (10*) to (11*);
    \draw[thick] (11*) to (12*);
    \draw[thick](12*) to (13*);
    \draw[thick](13*) to (14*);
    \draw[thick](14*) to (15*);
    \draw[thick](15*) to (16*);
    \draw[thick](16*) to (17*);
    \draw[thick](17*) to (18*);
    \draw[thick](18*) to (19*);    
\end{tikzpicture}
\hspace{.25in}
\scalebox{2.5}{\raisebox{12pt}{$\Leftrightarrow$}}
\hspace{.25in}
\raisebox{-30pt}{
\begin{tikzpicture}
	[scale=.55, auto=left, every node/.style = {circle, fill=black, inner sep=1.25pt}]
	\phantom{\draw[thick,dotted, color = gray] (0,0) to (5,0);}
	\draw[thick,dotted, color = gray] (0,3) to (6,3);
	\phantom{\draw[thick,dotted, color = gray] (0,6) to (5,6);}
	\node(1*) at (0,3) {};
    \node(2*) at (1,4) {};
    \node(3*) at (2,4) {};
    \node(4*) at (3,5) {};
    \node(5*) at (4,4) {};
    \node(6*) at (5,5) {};
    \node(7*) at (6,3) {};
    \draw[thick] (1*) to (2*);
    \draw[thick] (2*) to (3*);
    \draw[thick] (3*) to (4*);
    \draw[thick] (4*) to (5*);
    \draw[thick] (5*) to (6*);
    \draw[thick] (6*) to (7*);
\end{tikzpicture}} 

\vspace{-.3in}

\caption{An example of the bijection between $\mathcal{D}_{n,m}^{k,a}(0)$ and $M_{n,m}^{k-1}(\vec{\alpha},\vec{1})$ from the proof of Theorem \ref{3thm: peak parity 0}, here for $k=3$, $n=4$, $a= m = 0$.  In this case $\vec{\alpha} = \vec{0}$, so colors have been suppressed.}
\label{3fig: peak parity lemma example}
\end{figure}

\begin{corollary}
\label{3thm: peak parity 0 corollary}
Fix $k \geq 2$.  Then $\vert \mathcal{D}_{n,m}^{k,0} (0) \vert = M_{n,m}^{k-1}(\vec{0},\vec{1})$ for all $n \geq 0$ and $0 \leq m \leq n$, where $\vec{0} = (0,0,\hdots,0)$ and $\vec{1} = (1,1,\hdots,1)$. 
\end{corollary}

The primary insight in generalizing Callan's second identity is that $\mathcal{D}_{n,0}^{2,0}(1)$ lies in bijection with $\mathcal{D}_{n-1,0}^{2,1}(0)$.  This follows from the map that deletes the initial $U$ step and the final $D_1$ step of any $P \in \mathcal{D}_{n,0}^{2,0}(1)$, and then shifts the resulting path down by one.  Generalizing this map gives:

\begin{corollary}
\label{3thm: peak parity k-1}
Fix $k \geq 2$.  Then $\vert \mathcal{D}_{n,m}^{k,0} (k-1) \vert = M_{n-1,m}^{k-1}(\vec{1},\vec{1})$ for all $n \geq 1$ and $0 \leq m \leq n$, where $\vec{1} = (1,1,\hdots,1)$.
\end{corollary}
\begin{proof}
We establish a bijection $\phi : \mathcal{D}_{n,0}^{k,0}(k-1) \rightarrow \mathcal{D}_{n-1,0}^{k,k-1}(0)$.  For any $P \in \mathcal{D}_{n,0}^{k,0}(k-1)$, it must be the case that $P$ decomposes as $P = U^{k-1} P' D_{k-1}$.  Here $P'$ begins and ends at height $k-1$ and thus corresponds to some $Q \in \mathcal{D}_{n-1,0}^{k,k-1}$.  As $P$ has peak parity $k-1$, the path $Q$ must have peak parity $0$.  The map $\phi(P) = Q$ represents our desired bijection.

Applying Lemma \ref{3thm: peak parity 0} then gives $\vert \mathcal{D}_{n,0}^{k,0}(k-1) \vert = \vert \mathcal{D}_{n-1,0}^{k,k-1}(0) \vert = M_{n-1,0}^{k-1} (\vec{\alpha},\vec{1})$, where $\alpha_i = 1$ for all $i < k-1$ and thus for all coordinates $0 \leq i \leq k-2$.   
\end{proof}

The difficulty in generalizing Theorem \ref{3thm: peak parity 0} to peak parities other than $0$ or $k-1$ derives from the nature of our bijection between $\mathcal{D}_{n,m}^{k,a}$ and $\mathcal{M}_{n,m}^{k-1}(\vec{\alpha},\vec{\beta})$.  In particular, that bijection requires the Dyck paths have a length that is divisible by $kn$.  The decomposition technique of Corollary \ref{3thm: peak parity k-1} could be extended to paths of arbitrary peak parity $h$, but the required decomposition $P= U^h P' D_{k-1}$ only leaves a central subpath $P'$ with length divisible by $kn$ when $h=k-1$. 

\subsection{$(\vec{\alpha},\vec{\beta})$-Colored Motzkin Number and $k$-ary Trees}
\label{subsec: rooted k-ary trees}

For one final collection of combinatorial interpretations, we turn our attention to $k$-ary trees.  For any $k \geq 1$, a $k$-ary tree is a rooted tree in which every vertex has at most $k$ children.  A complete $k$-ary tree is a $k$-ary tree in which every vertex has either $0$ children or $k$ children.  We denote the set of all $k$-ary trees with precisely $n$ edges by $\mathcal{T}_n^k$, and the collection of all complete $k$-ary trees with precisely $n$ edges by $\mathcal{K}_n^k$.

It is well known that $2$-ary trees are enumerated by the Motzkin numbers as $\vert \mathcal{T}_n^2 \vert = M_{n,0}(1,1)$, and that complete $k$-ary trees are enumerated by the $k$-Catalan numbers as $\vert \mathcal{K}_{kn}^k \vert = C_n^k = M_{n,0}^k(\vec{0},\vec{0})$ for every $k \geq 2$.  See Aigner \cite{Aigner1} or Hilton and Pedersen \cite{HP} for bijections establishing these results.  These are the combinatorial interpretations that we look to generalize in this subsection.

In order to consistently describe our generalized bijection, we represent $k$-ary trees so that the root lies at the top of the tree and all children always appear lower than their parents.  One may then order the vertices of $T \in \mathcal{T}_n^k$ via a depth-first search, from left-to-right, and label each edge with one less than the integer assigned to the vertex at its bottom end.  See the left side of Figure \ref{3fig: tree bijection} for an example.  This edge ordering may be used to define a generalized bijection $\mathcal{T}_n^k$ and $\mathcal{M}_{n,0}^{k-1}(\vec{1},\vec{1})$:

\begin{proposition}
For any $n \geq 0$ and $k \geq 2$, $\vert \mathcal{T}_n^k \vert = M_{n,0}^{k-1}(\vec{1},\vec{1})$, where $\vec{1} = (1,1,\hdots,1)$.
\end{proposition}
\label{3thm: tree to Motzkin path bijection}
\begin{proof}
To define our bijection $\phi: \mathcal{T}_n^k \rightarrow \mathcal{M}_{n,0}^{k-1}(\vec{1},\vec{1})$, proceed through the edges of $T \in \mathcal{T}_n^k$ in the order defined by our depth-first search.  Then construct the path $\phi(T) \in \mathcal{M}_{n,0}^{k-1}(\vec{1},\vec{1})$ as follows, from left to right.  If a particular edge is not a rightmost child, append a $U$ step to the end of the partial path.  If an edge is a rightmost child, and if the vertex at its top end has precisely $i$ children, append a $D_{i-1}$ step to the end of the partial path.

The resulting path $\phi(T)$ clearly ends at $(n,0)$ and uses the correct step set for an element of $\mathcal{M}_{n,0}^{k-1}(\vec{1},\vec{1})$.  The fact that $\phi(T)$ remains weakly above $y=0$ is a consequence of the depth-first search: in our edge ordering, the children of any fixed vertex have labels that increase from left to right.  This means that the $D_{i-1}$ step associated with the rightmost child is always added after the $U$ steps of its $i-1$ non-rightmost siblings.

To see that $\phi$ is a bijection, notice that $T$ may be uniquely recovered from $\phi(T)$ as follows.  For every $D_i$ step in $\phi(T)$ with $i \geq 1$, identify the $i$ total $U$ steps that are ``visible" to that $D_i$ step from the left.  These matchings (which we represent via horizontal ``lasers" that travel under $\phi(T)$) correspond to siblings in the associated tree $T$.  Nesting of laser matchings in $\phi(T)$ correspond to parent/child relationships in $T$.  See the right side of Figure \ref{3fig: tree bijection} for an example of this inverse procedure.
\end{proof}

\begin{figure}[ht!]
\centering
\begin{tikzpicture}
    [scale=.45, auto=left, every node/.style = {circle, fill=black, inner sep=1.25pt}]
    \node(1*) at (4,6) {};
    \node(2*) at (2,4) {};
    \node(3*) at (4,4) {};
    \node(4*) at (6,4) {};
    \node(5*) at (0,2) {};
    \node(6*) at (2,2) {};
    \node(7*) at (6,2) {};
    \node(8*) at (2,0) {};
    \node(9*) at (4,0) {};
    \node(10*) at (6,0) {};
    \node(11*) at (8,0) {};
    \draw[thick] (1*) to (2*);
    \draw[thick] (1*) to (3*);
    \draw[thick] (1*) to (4*);
    \draw[thick] (2*) to (5*);
    \draw[thick] (2*) to (6*);
    \draw[thick] (4*) to (7*);
    \draw[thick] (6*) to (8*);
    \draw[thick] (7*) to (9*);
    \draw[thick] (7*) to (10*);
    \draw[thick] (7*) to (11*);
	\node[fill=none](1) at (2.55,5.25) {1};
	\node[fill=none](5) at (3.65,4.75) {5};
	\node[fill=none](6) at (5.35,5.25) {6};
	\node[fill=none](2) at (.6,3.25) {2};
	\node[fill=none](3) at (2.4,3) {3};
	\node[fill=none](7) at (5.6,3) {7};
	\node[fill=none](4) at (2.4,1) {4};
	\node[fill=none](8) at (4.65,1.25) {8};
	\node[fill=none](9) at (5.65,.65) {9};
	\node[fill=none](10) at (7.55,1.25) {10};	
\end{tikzpicture}
\hspace{.25in}
\scalebox{2.5}{\raisebox{12pt}{$\Leftrightarrow$}}
\hspace{.3in}
\raisebox{16pt}{
\begin{tikzpicture}
    [scale=.65, auto=left, every node/.style = {circle, fill=black, inner sep=1.25pt}]
    \draw[thick, dotted, color = black] (0,0) to (10,0);
    \node(0*) at (0,0) {};
    \node(1*) at (1,1) {};
    \node(2*) at (2,2) {};
    \node(3*) at (3,1) {};
    \node(4*) at (4,1) {};
    \node(5*) at (5,2) {};
    \node(6*) at (6,0) {};
    \node(7*) at (7,0) {};
    \node(8*) at (8,1) {};
    \node(9*) at (9,2) {};
    \node(10*) at (10,0) {};
	\draw[color=red, dotted, thick] (.5,.5) to (5.7,.5);
	\draw[color=red, dotted, thick] (1.5,1.5) to (2.5,1.5);
	\draw[color=red, dotted, thick] (4.5,1.5) to (5.25,1.5);
	\draw[color=red, dotted, thick] (7.5,.5) to (9.7,.5);
	\draw[color=red, dotted, thick] (8.5,1.5) to (9.25,1.5);  
    \draw[thick] (0*) to (1*);
    \draw[thick] (1*) to (2*);
    \draw[thick] (2*) to (3*);
    \draw[thick] (3*) to (4*);
    \draw[thick] (4*) to (5*);
    \draw[thick] (5*) to (6*);
    \draw[thick] (6*) to (7*);
    \draw[thick] (7*) to (8*);
    \draw[thick] (8*) to (9*);
    \draw[thick] (9*) to (10*);
\end{tikzpicture}}
\caption{An example of the bijection between $\mathcal{T}_n^k$ and $\mathcal{M}_{n,0}^{k-1}(\vec{1},\vec{1})$ from the proof of Theorem \ref{3thm: tree to Motzkin path bijection}.  The tree of the left side exhibits our depth-first edge ordering, whereas the dotted red lines on the right side correspond to the ``lasers" used in defining the inverse map.}
\label{3fig: tree bijection}
\end{figure}
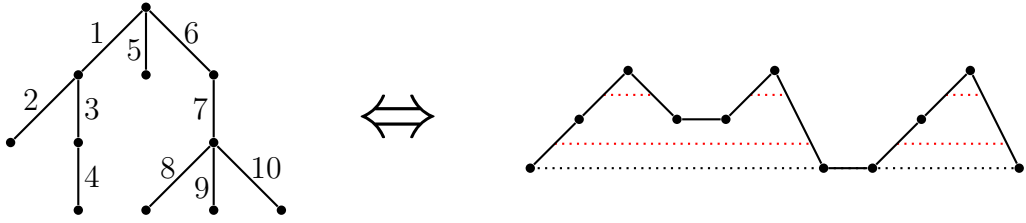

The bijection of Theorem \ref{3thm: tree to Motzkin path bijection} directly prompts a combinatorial interpretation for any sequence of $(\vec{\alpha},\vec{\beta})$-colored Motzkin numbers where the vectors $\vec{\alpha}$ and $\vec{\beta}$ are composed entirely of zeroes and ones:

\begin{corollary}
\label{3thm: trees and Motzkin paths corollary}
Fix $k \geq 2$, and let $S$ be an subset of $\lbrace 0,1,\hdots, k-1 \rbrace$. Then define $\mathcal{T}_n^{k,S} \subseteq \mathcal{T}_n^k$ to be the collection of all $k$-ary trees where every vertex must have either $k$ children or precisely $i$ children for some $i \in S$.  Then $\vert \mathcal{T}_n^{k,S} \vert = M_{n,0}^{k-1}(\vec{\alpha}_S,\vec{\alpha}_S)$, where $\vec{\alpha}_S = (\alpha_0,\hdots,\alpha_{k-1})$ is the $(k-1)$-tuple such that $\alpha_i = 1$ if $i \in S$ and $\alpha_i = 0$ otherwise.
\end{corollary}

\appendix

\section{Tables of $(\vec{\alpha},\vec{\beta})$-Colored Motzkin Numbers}
\label{sec: appendix}

A Java program was written that used Proposition \ref{2thm: higher-order Motzkin triangle as Riordan array} to generate the first seven rows of the $(\vec{\alpha},\vec{\beta})$-colored Motzkin triangle, for order $\ell=1,2,3$ and arbitrary choices of $\vec{\alpha},\vec{\beta}$.  The first columns of those Riordan arrays were then checked against OEIS \cite{OEIS} for pre-existing combinatorial interpretations.  The results of our comparisons are shown below for $\ell=1,2$ and various ``easy" choices of $\vec{\alpha},\vec{\beta}$.  Dashes correspond to sequences that failed to return an entry on OEIS.  Java code is available upon request.

\begin{table}[ht!]
\centering
\def\arraystretch{1.3}
\scalebox{.9}{
\begin{tabular}{|c|c|c|c|c|c|c|}
\hline
 & \textbf{$\boldsymbol{\beta=0}$} & \textbf{$\boldsymbol{\beta=1}$} & \textbf{$\boldsymbol{\beta=2}$} & \textbf{$\boldsymbol{\beta=3}$} & \textbf{$\boldsymbol{\beta=4}$} & \textbf{$\boldsymbol{\beta=5}$} \\ \hline
\textbf{$\boldsymbol{\alpha=0}$} & A126120 & A005043 & A000957 & A1177641 & A185132 & - \\ \hline
\textbf{$\boldsymbol{\alpha=1}$} & A001405 & A001006 & A000108 & A033321 & - & - \\ \hline
\textbf{$\boldsymbol{\alpha=2}$} & A054341 & A005773 & A000108* & A007317 & A033543 & - \\ \hline
\textbf{$\boldsymbol{\alpha=3}$} & A126931 & A059738 & A001700 & A002212 & A064613 & - \\ \hline
\textbf{$\boldsymbol{\alpha=4}$} & - & - & A049027 & A026378 & A005572 & A104455 \\ \hline
\textbf{$\boldsymbol{\alpha=5}$} & - & - & A076025 & - & A005573 & A182401 \\ \hline
\end{tabular}}
\caption{An expansion of Table \ref{1tab: Catalan-like numbers}, showing integer sequences corresponding to the $(\alpha,\beta)$-colored Motzkin numbers $M_{n,0}^1(\alpha,\beta)$ of order $\ell=1$, for various choices of $(\alpha,\beta)$.}
\label{appendix ell=1}
\end{table}

\begin{table}[ht!]
\centering
\def\arraystretch{1.3}
\scalebox{.9}{
\begin{tabular}{|c|c|c|c|c|c|c|}
\hline
 & $\boldsymbol{\beta=0}$ & $\boldsymbol{\beta=1}$ & $\boldsymbol{\beta=2}$ & $\boldsymbol{\beta=3}$ & $\boldsymbol{\beta=4}$ & $\boldsymbol{\beta=5}$ \\ \hline
$\boldsymbol{\alpha=0}$ & $\binom{n}{\lfloor \frac{n}{2} \rfloor}$ & A002426& A026641 & A126952 & - & - \\ \hline
$\boldsymbol{\alpha=1}$ & A000079 & A005773 & A000984 & A126568 & A227081 & - \\ \hline
$\boldsymbol{\alpha=2}$ & A127358 & A000244 & $\binom{2n+1}{n+1}$ & A026375 & A133158 & - \\ \hline
$\boldsymbol{\alpha=3}$ & A127359 & A126932  & A000302 & A026378 & A081671 & - \\ \hline
$\boldsymbol{\alpha=4}$ & A127360 & - & A141223 & - & A005573 & A098409 \\ \hline
$\boldsymbol{\alpha=5}$ & - & - & - & - & A000400 & A122898 \\ \hline
\end{tabular}}
\caption{Integer sequences $\lbrace r_n (\alpha,\beta) \rbrace_{n=0}^\infty$ corresponding to row sums $r_n = \sum_{i=0}^n M_{n,m}^1(\alpha,\beta)$ of the $(\alpha,\beta)$-Motzkin triangle of order $\ell=1$, for various choices of $(\alpha,\beta)$.  By Theorem \ref{2thm: row-sums of Motzkin triangle, color change version}, the $(i,i)$ entries of this table equal the $(i+1,i)$ entries of Table \ref{appendix ell=1}.}
\label{appendix ell=1 row sums}
\end{table}

\begin{table}[ht!]
\centering
\def\arraystretch{1.3}
\scalebox{.87}{
\begin{tabular}{|c|c|c|c|}
\hline
 & $\boldsymbol{\vec{\beta}=(0,0)}$ & $\boldsymbol{\vec{\beta}=(1,0)}$ & $\boldsymbol{\vec{\beta}=(2,0)}$ \\ \hline
$\boldsymbol{\vec{\alpha}=(0,0)}$ & - & - & - \\ \hline
$\boldsymbol{\vec{\alpha}=(1,0)}$ & A076227 & A071879 & - \\ \hline
$\boldsymbol{\vec{\alpha}=(2,0)}$ & - & - & - \\ \hline
\end{tabular}
\hspace{.04in}
\begin{tabular}{|c|c|c|c|}
\hline
 & $\boldsymbol{\vec{\beta}=(0,1)}$ & $\boldsymbol{\vec{\beta}=(1,1)}$ & $\boldsymbol{\vec{\beta}=(2,1)}$ \\ \hline
$\boldsymbol{\vec{\alpha}=(0,1)}$ & A001005 & - & A303730 \\ \hline
$\boldsymbol{\vec{\alpha}=(1,1)}$ & - & A036765 & A049128 \\ \hline
$\boldsymbol{\vec{\alpha}=(2,1)}$ & - & A159772 & - \\ \hline
\end{tabular}
}
\caption{Integer sequences corresponding to the $(\vec{\alpha},\vec{\beta})$-colored Motzkin numbers $M_{n,0}^2(\vec{\alpha},\vec{\beta})$ of order $\ell=2$, for various choices of $\vec{\alpha}=(\alpha_0,\alpha_1)$, $\vec{\beta}=(\beta_0,\beta_1)$ with $\alpha_0=\beta_0=0$ (left) and $\alpha_1=\beta_1=1$ (right).}
\label{appendix ell=2 01}
\end{table}

\begin{table}[ht!]
\centering
\def\arraystretch{1.3}
\scalebox{.9}{
\begin{tabular}{|c|c|c|c|c|c|}
\hline
 & $\boldsymbol{\alpha_1=0}$ & $\boldsymbol{\alpha_1=1}$ & $\boldsymbol{\alpha_1=2}$ & $\boldsymbol{\alpha_1=3}$ & $\boldsymbol{\alpha_1=4}$\\ \hline
$\boldsymbol{\alpha_0=0}$ & - & A089354 & A023053 & - & - \\ \hline
$\boldsymbol{\alpha_0=1}$ & - & - & A001764 & A121545 & - \\ \hline
$\boldsymbol{\alpha_0=2}$ & - & - & A098746 & A006013 & - \\ \hline
$\boldsymbol{\alpha_0=3}$ & - & - & - & A001764* & - \\ \hline
$\boldsymbol{\alpha_0=4}$ & - & - & - & A047099 & - \\ \hline
\end{tabular}}
\caption{Integer sequences corresponding to the $(\vec{\alpha},\vec{\beta})$-colored Motzkin numbers $M_{n,0}^2(\vec{\alpha},\vec{\beta})$ of order $\ell=2$, for various choices of $\vec{\alpha}=(\alpha_0,\alpha_1)$ when $\vec{\beta}=(3,3)$.  When $\ell=2$, observe that these cover all choices of $(\vec{\alpha},\vec{\beta})$ relevant to Subsections \ref{subsec: k-Dyck paths} and \ref{subsec: k-Fine paths}.}
\label{appendix ell=2 33}
\end{table}

\end{document}